\documentclass[12pt,a4paper,reqno]{amsart}
\usepackage{amssymb}
\usepackage{verbatim}
\usepackage{amscd}
\usepackage{enumerate}
\usepackage{graphicx}
\usepackage{siunitx}
\usepackage{color}
\numberwithin{equation}{section}

\usepackage{mathabx}

\usepackage{mathtools}%                  http://www.ctan.org/pkg/mathtools
\usepackage[tableposition=top]{caption}% http://www.ctan.org/pkg/caption
\usepackage{booktabs,dcolumn}%           http://www.ctan.org/pkg/dcolumn + http://www.ctan.org/pkg/booktabs

\DeclareCaptionLabelSeparator{separation}{:\quad}
\DeclareFontFamily{OT1}{rsfs}{}
\DeclareFontShape{OT1}{rsfs}{n}{it}{<-> rsfs10}{}
\DeclareMathAlphabet{\mathscr}{OT1}{rsfs}{n}{it}

\theoremstyle{plain}

\newtheorem{theorem}{Theorem}[section]
\newtheorem{proposition}[theorem]{Proposition}
\newtheorem{hypothesis}[theorem]{Hypothesis}
\newtheorem{lemma}[theorem]{Lemma}
\newtheorem{corollary}[theorem]{Corollary}
\newtheorem{conjecture}[theorem]{Conjecture}

\theoremstyle{definition}

\newtheorem{definition}[theorem]{Definition}
\newtheorem{remark}[theorem]{Remark}

\newtheorem{example}[theorem]{Example}

\newcommand\E{\mathbb{E}}
\newcommand\R{\mathbb{R}}
\newcommand\Z{\mathbb{Z}}
\newcommand\N{\mathbb{N}}

\newcommand\eps{\varepsilon}

\renewcommand\P{\mathbb{P}}

\renewcommand{\phi}{\varphi}
\begin{document}

\title{Sign patterns of the Liouville and M\"obius functions}

\author{Kaisa Matom\"aki}
\address{Department of Mathematics and Statistics \\
University of Turku, 20014 Turku\\
Finland}
\email{ksmato@utu.fi}

\author{Maksym Radziwi{\l}{\l}}
\address{Department of Mathematics \\
 Rutgers University, Hill Center for the Mathematical
Sciences \\
110 Frelinghuysen Rd., Piscataway, NJ 08854-8019}
\email{maksym.radziwill@gmail.com}

\author{Terence Tao}
\address{Department of Mathematics, UCLA\\
405 Hilgard Ave\\
Los Angeles CA 90095\\
USA}
\email{tao@math.ucla.edu}

\begin{abstract}  Let $\lambda$ and $\mu$ denote the Liouville and M\"obius functions respectively.  Hildebrand showed that all eight possible sign patterns for $(\lambda(n), \lambda(n+1), \lambda(n+2))$ occur infinitely often.
By using the recent result of the first two authors on mean values of multiplicative functions in short intervals, we strengthen Hildebrand's result by proving that each of these eight sign patterns occur with positive lower natural density.  We also obtain an analogous result for the nine possible sign patterns for $(\mu(n), \mu(n+1))$.  A new feature in the latter argument is the need to demonstrate that a certain random graph is almost surely connected.
\end{abstract}

\maketitle

%%%%%%%%%%%%%%%%%%%%%%%%%

\section{Introduction}

In this paper we strengthen some results on the sign patterns of the Liouville function $\lambda$, as well as obtain new results on the M\"obius function $\mu$.

We begin with the Liouville function.  It will be convenient (particularly in the combinatorial arguments used to prove our main theorems) to introduce the following notation.

\begin{definition}[Liouville sign pattern]  Let $k,l$ be non-negative integers, and let $n$ be an integer.  Let $\epsilon_{-k} \dots \epsilon_l$ be a string of $k+l+1$ symbols from the alphabet $\{+,-,\ast\}$.  We write
\begin{equation}\label{sign}
 n \mapsto \epsilon_{-k} \dots \stackrel{\vee}\epsilon_0 \dots \epsilon_l
\end{equation}
if $n > k$, $\lambda(n+i)=+1$ for all $-k \leq i \leq l$ with $\epsilon_i=+$, and $\lambda(n+i)=-1$ for all $-k \leq i \leq l$ with $\epsilon_i=-$.  We write the negation of \eqref{sign} as
$$
 n \not \mapsto \epsilon_{-k} \dots \stackrel{\vee}\epsilon_0 \dots \epsilon_l
$$
\end{definition}

The symbol $\vee$ is only present in the above notation as a positional marker (analogous to a decimal point in decimal notation) and has no further significance.

\begin{example}  The claim
$$ n \mapsto {-}{\ast}{\stackrel{\vee}{+}}{+}$$
is equivalent to the assertion that $n>2$, $\lambda(n-2)=-1$, $\lambda(n)=+1$, and $\lambda(n+1)=+1$, but makes no claim about the value of $\lambda(n-1)$.
\end{example}

In this notation, a well-known conjecture of Chowla \cite{chowla} can now be phrased as follows:

\begin{conjecture}[Chowla]\label{chow}  For any $k \geq 1$ and signs
 $\epsilon_1, \dots, \epsilon_k \in \{-1,+1\}$, the set of natural numbers $n$ for which
$$ n \mapsto \stackrel{\vee}\epsilon_1 \dots \epsilon_k $$
has natural density $\frac{1}{2^k}$ (that is, the density of this set in $[1,x]$ converges to $1/2^k$ in the limit $x \to \infty$).
\end{conjecture}

For $k=1$, this claim is equivalent to the prime number theorem, but for $k>1$ Chowla's conjecture remains open.  For $k \leq 3$, we have the following partial result of Hildebrand~\cite{hil}:

\begin{theorem}[Hildebrand]  For any $k=1,2,3$ and signs $\epsilon_1, \dots, \epsilon_k \in \{-,+\}$, the claim
$$ n \mapsto \stackrel{\vee}\epsilon_1 \dots \epsilon_k $$
occurs for infinitely many $n$.
\end{theorem}

Hildebrand's method was elementary, relying on an \emph{ad hoc} combinatorial analysis that relied primarily on the multiplicative properties of $\lambda$ at the small primes $2,3,5$.  In this paper, we combine the methods of Hildebrand with a recent result of the first two authors \cite{mr} to improve this result as follows.

\begin{definition}  A property $P(n)$ of an integer $n$ is said to hold with \emph{positive lower natural density} if
$$ \liminf_{x \to \infty} \frac{1}{x} \sum_{n \leq x: P(n)} 1> 0.$$
\end{definition}

\begin{theorem}[Liouville patterns of length three]\label{main}  For any $k=1,2,3$ and signs $\epsilon_1, \dots, \epsilon_k \in \{-,+\}$, the claim
$$ n \mapsto \stackrel{\vee}\epsilon_1 \dots \epsilon_k $$
occurs with positive lower natural density.
\end{theorem}

As with Hildebrand's arguments, our arguments extend to other completely multiplicative functions $f$ taking values in $-1,+1$ than $\lambda$, provided that $f$ agrees with $\lambda$ at the primes $2,3,5$ and obeys a prime number theorem in arithmetic progressions for any modulus dividing $60$.  We leave the details of this generalisation to the interested reader.

As it turns out, the most difficult sign patterns to handle for Theorem \ref{main} are $+++$ and $---$.  The problem is that the Liouville function $\lambda(n)$ could potentially behave like $f(n) \chi_3(n)$ for almost all $n$ that are not multiples of $3$, where $\chi_3$ is the primitive Dirichlet character of conductor $3$ (thus $\chi_3(3n+1)=1$ and $\chi_3(3n+2)=-1$ for all $n$) and $f: \N \to \{-1,+1\}$ is a function that changes sign very rarely.  In such a case, the sign patterns $+++$ and $---$ will almost never occur.  Fortunately, the results in \cite{mr} preclude this scenario; the main difficulty is then to show that this is essentially the \emph{only} scenario that could eliminate the $+++$ or $---$ patterns almost completely.

\begin{remark}  Strictly speaking, our arguments do not yield an explicit bound on the lower natural density, because we rely on Banach limits to simplify the presentation of the argument.  However, we believe that one could extract an effective lower bound on the density if required by avoiding the use of Banach limits, and keeping track of all error terms without passing to an asymptotic limit.
\end{remark}

\begin{remark}  In \cite{be} it was shown that $(\lambda(n), \lambda(n+r), \lambda(n+2r), \lambda(n+3r))$ attains all sixteen sign patterns in $\{-1,+1\}^4$ infinitely often, if $n$ ranges over the natural numbers and $r$ ranges over a bounded set.  It is plausible that using arguments similar to the ones here, one can show that these sixteen sign patterns also occur for $n$ in a set of positive lower density.  Note also from recent results on linear equations in the Liouville function (see \cite[Proposition 9.1]{gt-primes}, together with the companion results in \cite{gt-mobius} and \cite{gtz}) that for any fixed $k \geq 1$, the tuple $(\lambda(n),\dots,\lambda(n+(k-1)r))$ is asymptotically equidistributed in $\{-1,+1\}^k$ if $n,r$ range uniformly over (say) $[x,2x]$ for some large $x$ going to infinity; see also \cite{FH} for a recent generalization of these sorts of results to other linear forms and to more general bounded multiplicative functions.  It may be, in light of the results in \cite{mr}, that one can obtain a similar equidistribution result with $r$ restricted to a much smaller range that grows arbitrarily slowly with $x$, but we do not pursue this matter here. 
\end{remark}

We now turn to the M\"obius function $\mu$.  This function takes values in $\{-1,0,+1\}$ rather than $\{-1,1\}$, and the presence of the additional $0$ significantly complicates the analysis.  Nevertheless, we can show an analogue for Theorem \ref{main} for $k=1,2$ only:

\begin{theorem}[M\"obius patterns of length two]\label{main-2}  For any $k=1,2$ and $\epsilon_1, \dots, \epsilon_k \in \{-1,0,+1\}$, the claim
$$ (\mu(n),\dots,\mu(n+k-1)) = (\epsilon_1,\dots,\epsilon_k)$$
occurs with positive lower natural density.
\end{theorem}

The difficult case of this theorem occurs when $k=2$ and $\epsilon_1,\epsilon_2 \in \{+1,-1\}$.  Suppose for instance that we wanted to show that $(\mu(n),\mu(n+1)) = (+1,-1)$ occurred with positive lower natural density.  In the case of the Liouville function, one could show (as was done in \cite{mr}) that if the analogous claim $(\lambda(n),\lambda(n+1)) = (+1,-1)$ occurs with zero density, then there would often exist long chains $n,n+1,\dots,n+h$ of consecutive natural numbers on which $\lambda$ was constant, which is in contradiction with the results in \cite{mr}.  This argument no longer works in the case of the M\"obius function, due to the large number of zeroes of this function (for instance, the zeroes at the multiples of four).  To get around this, one needs to replace the chains $n,n+1,\dots,n+h$ by more complicated paths of nearby natural numbers, necessitating the analysis of the connectivity properties of a certain random graph, which will be done in Sections \ref{initial-sec}, \ref{conclude-sec}.  The further study of this random graph (or similar such graphs) may have some further applications; in particular, one may hope to use expansion properties of this graph to make progress towards the $k=2$ case of the Chowla conjecture (Conjecture \ref{chow}). In fact by pursuing this direction further, the third author has recently shown that a logarithmic form of Chowla's conjecture holds, and consequently that the logarithmic density of integers $n$ for which $(\mu(n), \mu(n+1)) = (\varepsilon_1, \varepsilon_2)$ exists (see \cite{Terry}).

\section{Asymptotic probability}\label{asym}

When dealing with natural densities, there is the minor technical difficulty that the lower and upper natural densities for a given set of integers need not match, leading to a breakdown of additivity of density in these situations.  To get around this problem we shall use the (artificial) device of Banach limits.

Let $P_0(n)$ be a property of the natural numbers that holds with zero lower natural density.  In particular, we can find a sequence $x_i$ of real numbers going to infinity such that
$$ \frac{1}{x_i} \sum_{n \leq x_i: P_0(n)} 1 \leq 2^{-i} $$
for all $i$.  This implies that
$$ \sum_{\frac{1}{i} x_i \leq n \leq x_i: P_0(n)} \frac{1}{n} \leq i 2^{-i}.$$
In particular
$$ \frac{1}{\log i} \sum_{\frac{1}{i} x_i \leq n \leq x_i: P_0(n)} \frac{1}{n} \to 0$$
as $i \to \infty$.  Henceforth we fix the sequence $x_i$ with this property.

Next, define a \emph{Banach limit} to be a linear functional $\operatorname{LIM}: \ell^\infty(\N) \to \R$ from bounded sequences $(a_i)_{i=1}^\infty$ of real numbers to the real numbers with the property that
$$ \liminf_{i \to \infty} a_i \leq \operatorname{LIM} (a_i)_{i=1}^\infty \leq \limsup_{i \to \infty} a_i$$
for all bounded sequences $(a_i)_{i=1}^\infty$; in particular, $\operatorname{LIM} (a_i)_{n=1}^\infty = \lim_{i \to\infty} a_i$ when $a_i$ is convergent.  As is well known, the existence of a Banach limit is guaranteed by the Hahn-Banach theorem, or by the existence of non-principal ultrafilters on the natural numbers.  Henceforth we will fix\footnote{Strictly speaking, this means that we are assuming the axiom of choice (or at least the ultrafilter lemma) in our arguments.  However, if desired, it is a routine matter to rewrite the arguments below in terms of limit superior and limit inferior instead of Banach limits and replacing additivity by subadditivity or superadditivity as appropriate, so that the axiom of choice is no longer required.  We leave the details of this modification of the argument to the interested reader.} a single Banach limit $\operatorname{LIM}$.

Given a property $P(n)$ of an integer $n$, we define the \emph{asymptotic probability} that $P(n)$ occurs to be the quantity
$$ \operatorname{LIM} \left( \frac{1}{\log i} \sum_{\frac{1}{i} x_i \leq n \leq x_i: P(n)} \frac{1}{n} \right)_{i=1}^\infty,$$
and say that $P(n)$ holds \emph{asymptotically almost surely} (or \emph{a.a.s.} for short) if its asymptotic probability is $1$.   Thus, for instance, with $\operatorname{LIM}$ and $x_i$ as above, the property $P_0(n)$ is asymptotically almost surely false.

Note that as $\operatorname{LIM}$ is linear, asymptotic probability obeys all the axioms of finitely additive probability.  For instance, if $P(n)$ and $Q(n)$ are properties with $Q$ holding asymptotically almost surely, then the asymptotic probability of $P(n)$ is equal to the asymptotic probability of $P(n) \wedge Q(n)$.

We will call a quantity \emph{fixed} if it does not depend on the parameter $i$.  Since $x_i \to \infty$, it is clear that for any fixed constant $C>0$, one has $n > C$ a.a.s..  (This does not contradict the finite nature of $n$, because our probability measure is only finitely additive, rather than countably additive.)

In view of the above construction, we see that Theorem \ref{main} follows from

\begin{theorem}\label{newmain}  Fix a sequence $x_i \to \infty$ and a Banach limit $\operatorname{LIM}$.  For any $k=1,2,3$ and signs $\epsilon_1, \dots, \epsilon_k \in \{-,+\}$, the claim
\begin{equation}\label{nak}
 n \mapsto \stackrel{\vee}\epsilon_1 \dots \epsilon_k 
\end{equation}
occurs with positive asymptotic probability.
\end{theorem}

Indeed, if \eqref{nak} occurs with zero natural density, one obtains a contradiction to Theorem \ref{newmain} after selecting $x_i$ as discussed in the start of the section, with $P_0$ taken to be the property \eqref{nak}.  Similarly, Theorem \ref{main-2} follows from

\begin{theorem}\label{newmain2}  Fix a sequence $x_i \to \infty$ and a Banach limit $\operatorname{LIM}$.  For any $k=1,2$ and $\epsilon_1, \dots, \epsilon_k \in \{-1,0,+1\}$, the claim
\begin{equation}\label{nak2}
(\mu(n),\dots,\mu(n+k-1)) = (\epsilon_1,\dots,\epsilon_k)
\end{equation}
occurs with positive asymptotic probability.
\end{theorem}

Henceforth the sequence $x_i \to \infty$ and Banach limit $\operatorname{LIM}$ will be fixed, and this fact will be omitted from the explicit formulation of all the propositions below.  (Thus, for instance, with this convention the first sentence of Theorems \ref{newmain} and \ref{newmain2} would be deleted.)

We now set out some basic rules for manipulating asymptotic probability, beyond the laws of finitely additive probability.  The first rule allows one to make linear changes of variable.

\begin{lemma}[Linear change of variable]\label{linear}  Let $P(n)$ be a property of integers $n$. Then for any natural number $q$ and integer $r$, the asymptotic probability of $P(qn+r)$ is equal to $q$ times the asymptotic probability of ``$P(n)$ and $n=r \hbox{ mod } q$''.  
\end{lemma}

Note that the $q=1$ case of Lemma \ref{linear} gives translation invariance: the asymptotic probability of $P(n)$ is equal to the asymptotic probability of $P(n+h)$ for any fixed $h$.  Another useful corollary of Lemma \ref{linear} is that if $r_1,\dots,r_q$ is any fixed set of representatives of residues modulo $q$, then $P(n)$ holds a.a.s. if and only if $P(qn+r)$ holds a.a.s. for each $r=r_1,\dots,r_q$.

\begin{proof}  Let $C$ be a large constant depending on $q,r$.  By deleting an event of asymptotic probability zero, we may assume that $P(n)$ only holds for $n > C$.  For sufficiently large $C$, this implies that if $P(qn+r)$ holds, then $\frac{1}{n} = (1+o(1)) \frac{q}{qn+r}$ where $o(1)$ goes to zero as $C \to \infty$, and thus on making the change of variables $m=qn+r$,
$$ \sum_{\substack{\frac{1}{i} x_i \leq n \leq x_i \\ P(qn+r)}} \frac{1}{n} = (1+o(1)) q \sum_{\substack{\frac{q}{i} x_i +r \leq m \leq qx_i + r \\ P(m) \hbox{ and } m = r \hbox{ mod } q}} \frac{1}{m}.$$
Also, we have
$$ \sum_{\frac{q}{i} x_i + r \leq m \leq qx_i + r: P(m) \hbox{ and } m=r \hbox{ mod } q} \frac{1}{m} = \sum_{\frac{1}{i}x_i \leq m \leq x_i: P(m) \hbox{ and } m=r \hbox{ mod } q} \frac{1}{m} + O_{C,q}(1)$$
for a quantity $O_{C,q}(1)$ bounded in magnitude by a function of $C$ and $q$.  Dividing by $\log i$, we conclude that
$$ \frac{1}{\log i} \sum_{\frac{1}{i} x_i \leq n \leq x_i: P(qn+r)} \frac{1}{n} = (1+o(1))q \frac{1}{\log i} \sum_{\frac{1}{i} x_i \leq n \leq x_i: P(n)} \frac{1}{n} + \frac{O_{C,q}(1)}{\log i},$$
and on taking Banach limits and then sending $C$ to infinity, we obtain the claim.
\end{proof}

Now we encode some known facts about the Liouville and M\"obius function in this language.  From the prime number theorem in arithmetic progressions we have
$$ \lim_{x \to \infty} \frac{1}{x} \sum_{n \leq x} \lambda(qn+r) = 0$$
for any $q \geq 1$ and $r \in \Z$,
and thus by summation by parts
$$ \lim_{x \to \infty} \frac{1}{\log x} \sum_{n \leq x} \frac{\lambda(qn+r)}{n} = 0.$$
Similarly for the M\"obius function.  We conclude:

\begin{proposition}\label{pnt}  Let $q \geq 1$ and $r \in \Z$.  Then the assertions $\lambda(qn+r)=+1$ and $\lambda(qn+r)=-1$ (or in the notation of this paper, $qn+r \mapsto \stackrel{\vee}+$ and $qn+r \mapsto \stackrel{\vee}-$) each occur with asymptotic probability $1/2$.  Also, the assertions $\mu(qn+r)=+1$ and $\mu(qn+r)=-1$ occur with equal asymptotic probability.
\end{proposition}

This immediately gives the $k=1$ case of Theorem \ref{newmain}.  Since $\mu^2$ has density 
$$\frac{1}{\zeta(2)} = \frac{6}{\pi^2} = 0.6079\dots,$$ 
we also know that $\mu^2(n)=1$ with asymptotic probability $\frac{1}{\zeta(2)}$, and hence by the above proposition the three events $\mu(n)=+1, \mu(n)=0, \mu(n)=-1$ occur with asymptotic probability $\frac{1}{2\zeta(2)}$, $1-\frac{1}{\zeta(2)}$, $\frac{1}{2\zeta(2)}$ respectively.  Thus we also obtain the $k=1$ case of Theorem \ref{newmain2}.

From Proposition \ref{pnt} and the Chinese remainder theorem, we also see that for any fixed $w$, the asymptotic probability that $\mu(n+1)=+1$ and $n$ is not divisible by $p^2$ for any $p \leq w$ is equal to the asymptotic probability that $\mu(n+1)=-1$ and $n$ is not divisible by $p^2$ for any $p \leq w$.  Taking limits as $w \to \infty$ (noting that the asymptotic probability that $n$ is divisible by $p^2$ for some $p>w$ is $O(1/w)$), one concludes that the the pair $(\mu(n),\mu(n+1))$ takes the values $(0,+1)$ and $(0,-1)$ with equal asymptotic probability, and similarly takes the values $(+1,0)$ and $(-1,0)$ with equal asymptotic probability.  Also, standard sieve theory arguments also show that the event $\mu^2(n)=\mu^2(n+1)=+1$ (which is excluding two residue classes modulo $p^2$ for each prime $p$) occurs with asymptotic probability
\begin{equation}\label{cdef}
 c := \prod_p \left(1 - \frac{2}{p^2}\right) = 0.3226\dots
\end{equation}
and hence by inclusion-exclusion, $(\mu(n),\mu(n+1))$ takes the value $(0,0)$ with asymptotic probability
$$ 1 - \frac{2}{\zeta(2)} + c = 0.1067\dots.$$
Further inclusion-exclusion then shows $(\mu(n),\mu(n+1))$ takes each of the four values $(+1,0)$, $(-1,0)$, $(0,+1)$, $(0,-1)$ with asymptotic probability
$$ \frac{1}{2} ( \frac{1}{\zeta(2)} - c ) = 0.1426\dots.$$
This gives all the cases of Theorem \ref{newmain2} except for those in which $k=2$ and $(\epsilon_1,\epsilon_2) = (+1,+1), (+1,-1), (-1,+1), (-1,-1)$, the treatment of which we defer to Sections \ref{Graph}-\ref{conclude-sec}.

Recently, the first two authors \cite{mr} established (among other things) that
$$ \limsup_{X \to \infty} \frac{1}{X} \sum_{X \leq x \leq 2X} \left|\frac{1}{h} \sum_{x \leq n \leq x+h} \lambda(n)\right| \leq c(h)$$
for any $h \geq 1$ some quantity $c(h)$ that goes to zero as $h \to \infty$, and similarly with $\lambda$ replaced by $\mu$.  From summation by parts this implies that
$$ \limsup_{x \to \infty} \frac{1}{\log x} \sum_{n \leq x} \frac{1}{n} \left|\frac{1}{h} \sum_{j=0}^h \lambda(n+j)\right| \leq c(h)$$

and thus

\begin{theorem}[Liouville or M\"obius in short intervals]\label{mrt}  For any $\eps > 0$ and any $h$ that is sufficiently large depending on $\eps$, one has
$$ \left |\sum_{j=0}^h \lambda(n+j)\right| \leq \eps h$$
with asymptotic probability at least $1-\eps$. Similarly with $\lambda$ replaced by $\mu$.
\end{theorem}

The above results also hold when the Liouville function is twisted by a fixed real Dirichlet character:

\begin{theorem}[Twisted Liouville or M\"obius in short intervals]\label{mrt-3}  Let $\chi$ be a fixed real Dirichlet character. For any $\eps > 0$ and any sufficiently large $h$, one has
$$  \left|\sum_{j=0}^h \lambda(n+j) \chi(n+j)\right| \leq \eps h$$
with asymptotic probability at least $1-\eps$.  Similarly with $\lambda$ replaced by $\mu$.
\end{theorem}

This result was recently extended to the case of complex Dirichlet characters in \cite[Appendix A]{mrt}, but we will not need that extension here.  Indeed, we will only need Theorem \ref{mrt-3} with $\chi$ the non-trivial character $\chi_3$ of period $3$.

As already essentially observed in \cite{mr}, Theorem \ref{mrt} gives the $k=2$ case of Theorem \ref{newmain}:

\begin{theorem}\label{pap} The assertions $n \mapsto {\stackrel{\vee}{+}}{-}$, $n \mapsto {\stackrel{\vee}{-}}{+}$ hold with positive asymptotic probability, and the assertions $n \mapsto {\stackrel{\vee}{+}}{+}$, $n \mapsto {\stackrel{\vee}{-}}{-}$ hold with asymptotic probability at least $1/6$ each.
\end{theorem}

We remark that the second part of this theorem is essentially due to Harman, Pintz, and Wolke \cite{hpw}.

\begin{proof}  We begin with the sign pattern ${\stackrel{\vee}{+}}{-}$.  Assume for contradiction that $n \mapsto {\stackrel{\vee}{+}}{-}$ holds with zero asymptotic probability.  By Proposition \ref{pnt} and Lemma \ref{linear}, the claims $n \mapsto {\stackrel{\vee}{+}}{\ast}$ and $n \mapsto {\stackrel{\vee}{\ast}}{-}$ hold with asymptotic probability $1/2$, so that ${\stackrel{\vee}{+}}{+}$ and ${\stackrel{\vee}{-}}{-}$ hold with asymptotic probability $1/2$. Hence $n \mapsto {\stackrel{\vee}{-}}{+}$ holds with zero asymptotic probability.  In other words, we have $\lambda(n+1)=\lambda(n)$ a.a.s., hence by Lemma \ref{linear} and finite additivity we have for any fixed $h$ that $\lambda(n) = \lambda(n+1) = \dots = \lambda(n+h)$ a.a.s..  But this contradicts Theorem \ref{mrt} if $h$ is large enough.

The same argument also works for the sign pattern ${\stackrel{\vee}{-}}{+}$.  Now we consider the sign pattern ${\stackrel{\vee}{+}}{+}$.  From Proposition \ref{pnt} and inclusion-exclusion, the sign pattern ${\stackrel{\vee}{-}}{-}$ occurs with the same asymptotic probability as ${\stackrel{\vee}{+}}{+}$.  Thus it suffices to show that $\lambda(n)=\lambda(n+1)$ holds with asymptotic probability at least $1/3$.  But from the pigeonhole principle, at least one of $\lambda(2n+1)=\lambda(2n)$, $\lambda(2n+2)=\lambda(2n+1)$, and $\lambda(2n)=\lambda(2n+2)$ must hold for any $n$, which implies that the asymptotic probabilities of $\lambda(2n+1)=\lambda(2n)$, $\lambda(2n+2)=\lambda(2n+1)$, and $\lambda(n+1)=\lambda(n)$ add up to at least $1$.  On the other hand, from Lemma \ref{linear} we see that the asymptotic probability of $\lambda(n+1)=\lambda(n)$ is the average of the asymptotic probabilities of $\lambda(2n+1)=\lambda(2n)$ and $\lambda(2n+2)=\lambda(2n+1)$, and the claim follows.
\end{proof}

As essentially already observed by Hildebrand \cite{hil}, a simple translation argument then gives four of the eight subcases of the $k=3$ case of Theorem \ref{newmain}:

\begin{corollary} The assertions $n \mapsto {\stackrel{\vee}{+}}{+}{-}$, $n \mapsto {\stackrel{\vee}{-}}{+}{+}$, $n \mapsto {\stackrel{\vee}{+}}{-}{-}$, $n \mapsto {\stackrel{\vee}{-}}{-}{+}$ each hold with positive asymptotic probability.
\end{corollary}

\begin{proof}  We show this for the sign pattern ${\stackrel{\vee}{+}}{+}{-}$ only, as the other three cases are similar.  Suppose for contradiction that $n \not \mapsto {\stackrel{\vee}{+}}{+}{-}$ a.a.s..  In particular, $n \mapsto {\stackrel{\vee}{+}}{+}$ implies $n+1 \mapsto {\stackrel{\vee}{+}}{+}$ a.a.s..  Iterating this (using the translation invariance from Lemma \ref{linear}), we see that for any fixed $h \geq 1$, $n \mapsto {\stackrel{\vee}{+}}{+}$ implies that $\lambda(n) =\lambda(n+1)=\dots =\lambda(n+h)=+1$ a.a.s., and hence by Theorem \ref{pap}, one has $\lambda(n) =\lambda(n+1)=\dots =\lambda(n+h)=+1$ with asymptotic probability at least $c$ for some $c>0$ independent of $h$.  But this contradicts Theorem \ref{mrt} if $h$ is chosen sufficiently large.
\end{proof}

The same argument, in combination with the remaining $k=3$ cases of Theorem \ref{newmain} that we will handle shortly, show that the $k=4$ case of Theorem \ref{newmain} holds for the sign patterns ${\stackrel{\vee}{+}}{+}{+}{-}$, ${\stackrel{\vee}{-}}{+}{+}{+}$, ${\stackrel{\vee}{-}}{-}{-}{+}$, and ${\stackrel{\vee}{+}}{-}{-}{-}$.  However, there does not seem to be a similarly simple argument to handle the remaining twelve sign patterns of length $4$.  For longer patterns, we have the following partial results (the first of which relies on Theorem \ref{newmain}).  For technical reasons, these partial results seem to be limited to the category of positive upper density rather than positive lower density.

\begin{proposition}  Let $k \geq 3$.  
\begin{itemize}
\item[(i)] There are at least $k+5$ sign patterns $\epsilon_1 \dots \epsilon_k \in \{-,+\}^k$ such that $n \mapsto {\stackrel{\vee}{\epsilon_1}}\dots{\epsilon_k}$ occurs with positive upper density.
\item[(ii)] If $k$ is even, then there exists a sign pattern $\epsilon_1 \dots \epsilon_k \in \{-,+\}^k$ with exactly $k/2$ $+$ signs and $k/2$ $-$ signs such that $n \mapsto {\stackrel{\vee}{\epsilon_1}}\dots{\epsilon_k}$ occurs with positive upper density.
\end{itemize}
\end{proposition}

\begin{proof}  We claim that it suffices to establish (i) and (ii) with ``positive upper density'' replaced by ``positive asymptotic probability''.  Suppose for instance that the claim (i) failed, then there is a subset $S$ of $\{-,+\}^k$ consisting of at most $k+4$ sign patterns such that $n \mapsto {\stackrel{\vee}{\epsilon_1}}\dots{\epsilon_k}$ occurs with zero upper density for any $\epsilon_1 \dots \epsilon_k$ outside of $S$.  As upper density is subadditive (in contrast to lower density, which is superadditive), we conclude that outside of a set of zero upper density (and hence zero lower density), one has $n \mapsto {\stackrel{\vee}{\epsilon_1}}\dots{\epsilon_k}$ for some sign pattern in $S$.  Taking Banach limits as before, we can locate a sequence $x_i$ and a Banach limit such that $n \mapsto {\stackrel{\vee}{\epsilon_1}}\dots{\epsilon_k}$ for some sign pattern in $S$ a.a.s..  Taking contrapositives, we see that (i) follows from the asymptotic probability version of (i).  The argument for (ii) is similar and is left to the reader.

We now prove (i) (with ``positive upper density'' replaced by ``positive asymptotic probability'') by induction on $k$.  The case $k=3$ follows from Theorem \ref{newmain}.  Now suppose for sake of contradiction that $k \geq 4$, and that  $n \mapsto {\stackrel{\vee}{\epsilon_1}}\dots{\epsilon_{k-1}}$ with positive asymptotic probability for at least $k+4$ sign patterns of length $k-1$, but that  $n \mapsto {\stackrel{\vee}{\epsilon_1}}\dots{\epsilon_{k}}$ with positive asymptotic probability for at most $k+4$ sign patterns of length $k$.  Since each sign pattern of length $k-1$ that occurs with positive asymptotic probability has to have at least one extension to a sign pattern of length $k$ that occurs with positive asymptotic probability, we conclude that there must exist a function $\epsilon_{k}: \{-,+\}^{k-1} \to \{-,+\}$ such that
$$ n \mapsto {\stackrel{\vee}{\epsilon_1}}\dots{\epsilon_{k-1}} \implies n \mapsto {\stackrel{\vee}{\epsilon_1}}\dots{\epsilon_{k-1}}{\epsilon_{k}(\epsilon_1,\dots,\epsilon_{k-1})}$$
occurs a.a.s..  The map $(\epsilon_1,\dots,\epsilon_{k-1}) \mapsto (\epsilon_2,\dots,\epsilon_{k-1},\epsilon_{k}(\epsilon_1,\dots,\epsilon_{k-1}))$ on the finite set $\{-,+\}^{k-1}$ is clearly periodic, and hence on iterating the above implication (and using the translation invariance provided by Lemma \ref{linear}) there exists a natural number $q$ such that
$$ \lambda(n+q) = \lambda(n) $$
a.a.s..  By Lemma \ref{linear} this implies that $\lambda(qn+q)=\lambda(qn)$ a.a.s., and hence by multiplicativity $\lambda(n+1)=\lambda(n)$ a.a.s., but this contradicts Theorem \ref{pap}.

Now we prove (ii).  By the pigeonhole principle, the claim is equivalent to showing that
$$ \sum_{i=0}^{k-1} \lambda(n+i) = 0$$
with positive asymptotic probability.  Suppose to the contrary that
$$ \sum_{i=0}^{k-1} \lambda(n+i) \neq 0$$
a.a.s..  Observe that the sum $\sum_{i=0}^{k-1} \lambda(n+i)$ is always even, and changes by at most $2$ when one advances $n$ to $n+1$.  In particular, the sum $\sum_{i=0}^{k-1} \lambda(n+i)$ does not change sign a.a.s. when advancing from $n$ to $n+1$.  Iterating this observation (using the translation invariance provided by Lemma \ref{linear}), we conclude that for any natural number $H$, it is a.a.s. true that the sums $\sum_{i=0}^{k-1} \lambda(n+i+h)$ for $h=0,\dots,H-1$ are all either at least $+2$, or at most $-2$.  In particular, on summing in $h$, we conclude that the expression
$$ k \sum_{i=0}^{k+H-1} \lambda(n+i) $$
is at least $2H - 2k$, or at most $-2H + 2k$, a.a.s..  But this contradicts Theorem \ref{mrt} if $H$ is large enough.
\end{proof}

Returning to Theorem \ref{newmain},
it remains to verify the $k=3$ cases for the sign patterns $n \mapsto {\stackrel{\vee}{+}}{+}{+}$, $n \mapsto {\stackrel{\vee}{-}}{-}{-}$, $n \mapsto {\stackrel{\vee}{+}}{-}{+}$, $n \mapsto {\stackrel{\vee}{-}}{+}{-}$.  We now address these cases in the next two sections.

\section{The ${\stackrel{\vee}{+}}{+}{+}$ case}

We now verify Theorem \ref{newmain} for the sign pattern $n \mapsto {\stackrel{\vee}{+}}{+}{+}$.  The same argument (flipping all the signs) applies to the sign pattern $n \mapsto {\stackrel{\vee}{-}}{-}{-}$, and is left to the reader\footnote{Alternatively, one can observe that the only properties of $\lambda$ used in these arguments are those given by Proposition \ref{pnt}, Theorem \ref{mrt}, Theorem \ref{mrt-3}, that $\lambda$ takes values in $\{-1,+1\}$, and that $\lambda(pn)=-\lambda(n)$ for all $n$ and $p=2,3,5$.  All of these properties hold with $\lambda$ replaced by $-\lambda$, and so the arguments in this section applied to $-\lambda$ handle the sign pattern $n \mapsto {\stackrel{\vee}{-}}{-}{-}$.  Similarly for the arguments in the next section.}.

Throughout this section we assume for sake of contradiction that Theorem \ref{newmain} fails for $n \mapsto {\stackrel{\vee}{+}}{+}{+}$.  Thus we have

\begin{hypothesis}\label{po} We have $n \not \mapsto {\stackrel{\vee}{+}}{+}{+}$ a.a.s..
\end{hypothesis}

As remarked in the introduction, Hypothesis \ref{po} is morally compatible with $\lambda$ ``pretending'' to be like the Dirichlet character $\chi_3$ of conductor $3$, defined by setting $\chi_3(3n+1)=+1$, $\chi_3(3n+2)=-1$, $\chi_3(3n)=0$ for any $n$.  

The strategy will be to leverage Hypothesis \ref{po}, together with Lemma \ref{linear} and the multiplicative nature of $\lambda$, to force $\lambda$ to behave increasingly like $\chi_3$ in various senses, until we can use Theorem \ref{mrt-3} to obtain a contradiction.

We first give some simple consequences of Hypothesis \ref{po} and Lemma \ref{linear}.

\begin{corollary}\label{easy}  Asymptotically almost surely in $n$, the following claims hold.
\begin{itemize}
\item[(a)]  If $n \mapsto {\stackrel{\vee}{+}}{\ast}{+}$, then $n \mapsto {\stackrel{\vee}{+}}{-}{+}$.
\item[(a.2)]  If $2n \mapsto {\stackrel{\vee}{-}}{\ast}{\ast}{\ast}{-}$, then $2n \mapsto {\stackrel{\vee}{-}}{\ast}{+}{\ast}{-}$.
\item[(b)]  If $n \mapsto {\stackrel{\vee}{+}}{+}$, then $n \mapsto {-}{\stackrel{\vee}{+}}{+}{-}$.
%\item[(b.2)]  If $2n \mapsto \stackrel{\vee}-\ast-$, then $2n \mapsto +\ast\stackrel{\vee}-\ast-\ast+$.
\item[(b.3)]  If $3n \mapsto {\stackrel{\vee}{-}}{\ast}{\ast}{-}$, then $3n \mapsto {+}{\ast}{\ast}{\stackrel{\vee}{-}}{\ast}{\ast}{-}{\ast}{\ast}{+}$.
%\item[(b.4)]  If $4n \mapsto \stackrel{\vee}+\ast\ast\ast+$, then $4n \mapsto -\ast\ast\ast\stackrel{\vee}+\ast\ast\ast+\ast\ast\ast-$.
\item[(b.5)]  If $5n \mapsto {\stackrel{\vee}{-}}{\ast}{\ast}{\ast}{\ast}{-}$, then $5n \mapsto {+}{\ast}{\ast}{\ast}{\ast}{\stackrel{\vee}{-}}{\ast}{\ast}{\ast}{\ast}{-}{\ast}{\ast}{\ast}{\ast}{+}$.
\item[(c)]  If $3n \mapsto {\stackrel{\vee}{\ast}}{+}{+}$, then $3n \mapsto {+}{\ast}{\ast}{\stackrel{\vee}{-}}{+}{+}{-}{\ast}{\ast}{+}$.
\end{itemize}
\end{corollary}

\begin{proof}  The claim (a) is immediate from Hypothesis \ref{po}.  As $\lambda(2m)=-\lambda(m)$ for all $m$, that the claim (a.2) is equivalent to (a).  

For (b), observe from Hypothesis \ref{po} and Lemma \ref{linear} that $n \not \mapsto {\stackrel{\vee}{+}}{+}{+}$ and $n \not \mapsto {+}{\stackrel{\vee}{+}}{+}$ a.a.s., giving the claim.  The multiplicativity properties $\lambda(3m)=-\lambda(m)$ and $\lambda(5m)=-\lambda(m)$ then give the claims (b.3) and (b.5) respectively.

Now we prove (c).  Suppose $3n \mapsto {\stackrel{\vee}{\ast}}{+}{+}$.  From (b) and Lemma \ref{linear}, we then have $3n \mapsto {\stackrel{\vee}{-}}{+}{+}{-}$ a.a.s..  The claim then follows from (b.3).
\end{proof}

The next implication is essentially due to Hildebrand \cite{hil}.

\begin{proposition}\label{15n}  One has $15 n \not \mapsto {+}{\stackrel{\vee}{\ast}}{+}$ a.a.s..
\end{proposition}

\begin{proof}  We will restrict to the event $15 n \mapsto {+}{\stackrel{\vee}{\ast}}{+}$ and show that this leads a.a.s. to a contradiction, giving the claim.

By hypothesis and Corollary \ref{easy}(a) we have $15n \mapsto {+}{\stackrel{\vee}{-}}{+}$ a.a.s..  Since $\lambda(4m)=\lambda(m)$, we thus have
$$ 60n \mapsto {+}{\ast}{\ast}{\ast}{\stackrel{\vee}{-}}{\ast}{\ast}{\ast}{+}$$
a.a.s..  Clearly this implies
$$ 60n \not \mapsto {\stackrel{\vee}{+}}{\ast}{\ast}{-}{+}{+}{-}{\ast}{\ast}{+}$$
so by Corollary \ref{easy}(c) in the contrapositive (and Lemma \ref{linear}) we have
$$ 60n \not \mapsto {\stackrel{\vee}{\ast}}{\ast}{\ast}{\ast}{+}{+}$$
and thus
$$ 60n \mapsto {+}{\ast}{\ast}{\ast}{\stackrel{\vee}{-}}{\ast}{\ast}{\ast}{+}{-}$$
a.a.s..  A similar argument gives
$$ 60n \not \mapsto {+}{+}{\ast}{\ast}{\ast}{\stackrel{\vee}{\ast}}$$
and thus
$$ 60n \mapsto {-}{+}{\ast}{\ast}{\ast}{\stackrel{\vee}{-}}{\ast}{\ast}{\ast}{+}{-}$$
a.a.s..  But from Corollary \ref{easy}(b.5) and Lemma \ref{linear}, we must then have
$$ 60n \mapsto {+}{\ast}{\ast}{\ast}{\ast}{\stackrel{\vee}{-}}{\ast}{\ast}{\ast}{\ast}{-}$$
a.a.s., giving the desired contradiction since $\lambda(60n-5)$ cannot simultaneously be $+1$ and $-1$.  
\end{proof}

This gives us our first step towards demonstrating that $\lambda$ ``pretends to be like'' $\chi_3$:

\begin{corollary}\label{15c}  We have $\lambda(15n+1)=-\lambda(15n-1)$ a.a.s..
\end{corollary}

\begin{proof}  By Proposition \ref{pnt}, we have $\lambda(15n+1)=+1$ and $\lambda(15n-1)=+1$ with an asymptotic probability of $1/2$, while from Proposition \ref{15n} we have $\lambda(15n+1)=\lambda(15n-1)=+1$ with asymptotic probability zero.  The claim then follows from the inclusion-exclusion principle.
\end{proof}

A variant of the above arguments gives

\begin{proposition}\label{6n}  One has $6n \not \mapsto {-}{\stackrel{\vee}{+}}{-}$ a.a.s..
\end{proposition}

\begin{proof}  As before, we restrict to the event $6n \mapsto {-}{\stackrel{\vee}{+}}{-}$ and show that this leads to a contradiction a.a.s..  

The multiplicativity property $\lambda(5m)=-\lambda(m)$ for all $m$ gives
$$ 30 n \mapsto {+}{\ast}{\ast}{\ast}{\ast}{\stackrel{\vee}{-}}{\ast}{\ast}{\ast}{\ast}{+}.$$
In particular,
$$ 30 n \not \mapsto {\stackrel{\vee}{+}}{\ast}{\ast}{-}{+}{+}{-}{\ast}{\ast}{+}$$
and hence from Corollary \ref{easy}(c) in the contrapositive (and Lemma \ref{linear}) we have
$$ 30 n \not \mapsto {\stackrel{\vee}{\ast}}{\ast}{\ast}{\ast}{+}{+}$$
and hence
$$ 30 n \mapsto {+}{\ast}{\ast}{\ast}{\ast}{\stackrel{\vee}{-}}{\ast}{\ast}{\ast}{-}{+}$$
a.a.s..  A similar argument gives
$$ 30 n \not \mapsto {+}{+}{\ast}{\ast}{\ast}{\stackrel{\vee}{\ast}}$$
and thus
$$ 30 n \mapsto {+}{-}{\ast}{\ast}{\ast}{\stackrel{\vee}{-}}{\ast}{\ast}{\ast}{-}{+}$$
a.a.s.  By two applications of Corollary \ref{easy}(a.2) and Lemma \ref{linear}, we then have
$$ 30 n \mapsto {+}{-}{\ast}{+}{\ast}{\stackrel{\vee}{-}}{\ast}{+}{\ast}{-}{+}$$
a.a.s., and hence since $\lambda(2m)=-\lambda(m)$
$$ \lambda(15n-1) = \lambda(15n+1) = -1,$$
but this contradicts Corollary \ref{15c} a.a.s..
\end{proof}

\begin{corollary}\label{3n} One has $3n \not \mapsto {-}{\stackrel{\vee}{-}}{-}$ a.a.s..
\end{corollary}

\begin{proof}  We restrict to the event that $3n \mapsto {-}{\stackrel{\vee}{-}}{-}$.  The multiplicativity property $\lambda(2m)=-\lambda(m)$ then gives
$$ 6n \mapsto {+}{\ast}{\stackrel{\vee}{+}}{\ast}{+}$$
and hence by two applications of Corollary \ref{easy}(a) and Lemma \ref{linear} we have
$$ 6n \mapsto {+}{-}{\stackrel{\vee}{+}}{-}{+}$$
a.a.s..  But this contradicts Proposition \ref{6n} a.a.s..
\end{proof}

\begin{corollary}\label{corp}  Asymptotically almost surely, we have $\lambda(6n-1)=-\lambda(6n+1)$.
\end{corollary}

\begin{proof}  By Proposition \ref{pnt},  we have $\lambda(6n-1)=-1$ and $\lambda(6n+1)=-1$ with an asymptotic probability of $1/2$ each. From Proposition \ref{6n}, Corollary \ref{3n}, and Lemma \ref{linear} we see that $\lambda(6n-1)=\lambda(6n+1)=-1$ holds with asymptotic probability zero.  The claim then follows from the inclusion-exclusion principle.
\end{proof}

Next, we work on improving Corollary \ref{easy}(c).

\begin{proposition}\label{9a}  One has $9n+3 \not \mapsto {\stackrel{\vee}{\ast}}{+}{+}$ a.a.s..
\end{proposition}

\begin{proof}  We restrict to the event that $9n+3 \mapsto {\stackrel{\vee}{\ast}}{+}{+}$.  By Corollary \ref{easy}(c) and Lemma \ref{linear} we then have
$$ 9n+3 \mapsto {+}{\ast}{\ast}{\stackrel{\vee}{-}}{+}{+}{-}{\ast}{\ast}{+}$$
a.a.s..  Since $\lambda(2m)=-\lambda(m)$, we then have
$$ 18n+9 \mapsto {+}{\ast}{-}{\stackrel{\vee}{\ast}}{-}{\ast}{+}.$$
By Corollary \ref{3n} and Lemma \ref{linear} we thus have
$$ 18n+9 \mapsto {+}{\ast}{-}{\stackrel{\vee}{+}}{-}{\ast}{+}$$
a.a.s.; since $\lambda(3m) = -\lambda(m)$, we then have
$$ 6n+3 \mapsto {-}{\stackrel{\vee}{-}}{-}$$
which by Corollary \ref{3n} and Lemma \ref{linear} leads to a contradiction a.a.s..
\end{proof}

\begin{proposition}\label{9b}  Asymptotically almost surely, the claim $9n \mapsto {\stackrel{\vee}{\ast}}{+}{+}$ implies $6n-3 \mapsto {\stackrel{\vee}{\ast}}{+}{+}$, and similarly $9n-3 \mapsto {\stackrel{\vee}{\ast}}{+}{+}$ implies $6n \mapsto {\stackrel{\vee}{\ast}}{+}{+}$.
\end{proposition}

\begin{proof}  Let us first restrict to the event that $9n \mapsto {\stackrel{\vee}{\ast}}{+}{+}$.  By Corollary \ref{easy}(c) and Lemma \ref{linear}, we then have
$$ 9n \mapsto {+}{\ast}{\ast}{\stackrel{\vee}{-}}{+}{+}{-}{\ast}{\ast}{+}$$
a.a.s..  Since $\lambda(2m)=-\lambda(m)$, this implies
$$ 18n \mapsto {-}{\ast}{\ast}{\ast}{\ast}{\ast}{\stackrel{\vee}{+}}{\ast}{-}{\ast}{-}.$$
From Corollary \ref{3n} and Lemma \ref{linear} we a.a.s. have
$$ 18n \not \mapsto {\stackrel{\vee}{\ast}}{\ast}{-}{-}{-}$$
and hence
$$ 18n \mapsto {-}{\ast}{\ast}{\ast}{\ast}{\ast}{\stackrel{\vee}{+}}{\ast}{-}{+}{-};$$
since $\lambda(3m)=-\lambda(m)$, this implies that
$$ 6n-3 \mapsto {\stackrel{\vee}{\ast}}{+}{\ast}{-}{-}.$$
But from Corollary \ref{3n} and Lemma \ref{linear} we have
$$ 6n-3 \not \mapsto {\stackrel{\vee}{\ast}}{+}{-}{-}{-}$$
a.a.s., and hence $6n-3 \mapsto {\stackrel{\vee}{\ast}}{+}{+}$ a.a.s. as required.

Similarly, if we instead restrict to the event $9n-3 \mapsto {\stackrel{\vee}{\ast}}{+}{+}$, then Corollary \ref{easy}(c) and Lemma \ref{linear} give
$$ 9n \mapsto {+}{\ast}{\ast}{-}{+}{+}{\stackrel{\vee}{-}}{\ast}{\ast}{+}$$
and then
$$ 18n \mapsto {-}{\ast}{-}{\ast}{\stackrel{\vee}{+}}{\ast}{\ast}{\ast}{\ast}{\ast}{-}$$
and then by Corollary \ref{3n} and Lemma \ref{linear} as before
$$ 18n \mapsto {-}{+}{-}{\ast}{\stackrel{\vee}{+}}{\ast}{\ast}{\ast}{\ast}{\ast}{-}$$
and thus
$$ 6n \mapsto {-}{\stackrel{\vee}{-}}{\ast}{+}$$
and then by Corollary \ref{3n} and Lemma \ref{linear} we have $6n \mapsto {\stackrel{\vee}{\ast}}{+}{+}$ as required.
\end{proof}

By combining the two preceding propositions with an iteration argument, we obtain

\begin{corollary}\label{3p}  We have $3n \not \mapsto {\stackrel{\vee}{\ast}}{+}{+}$ a.a.s..
\end{corollary}

\begin{proof}  Let $f: \N \to \N \cup \{\bot\}$ be the partially defined function given by the formulae $f(3n) := 2n-1$, and $f(3n-1) := 2n$ for $n \geq 2$, with all other values of $f$ equal to the undefined symbol $\bot$.  We can then combine Propositions \ref{9a}, \ref{9b} to give the following assertion (by dividing into cases based on the residue of $n$ modulo $3$): if $3n \mapsto {\stackrel{\vee}{\ast}}{+}{+}$, then a.a.s. $f(n) \neq \bot$ and $3f(n) \mapsto {\stackrel{\vee}{\ast}}{+}{+}$. Iterating this, we conclude in particular that for any fixed $k$, we have a.a.s. that if $3n \mapsto {\stackrel{\vee}{\ast}}{+}{+}$, then the sequence $n, f(n), f^2(n), \dots, f^k(n)$ avoids $\bot$.  But a routine count shows that the event that $n, f(n), f^2(n), \dots, f^k(n)$ avoids $\bot$ occurs with asymptotic probability $(2/3)^{k+1}$.  As $k$ can be arbitrarily large, we obtain the claim.
\end{proof}

\begin{corollary}\label{312}  Asymptotically almost surely, we have $\lambda(3n+1)=-\lambda(3n+2)$.
\end{corollary}

\begin{proof}  By Proposition \ref{pnt},  we have $\lambda(3n+1)=+1$ and $\lambda(3n+2)=+1$ with an asymptotic probability of $1/2$ each. From Corollary \ref{3p} we see that $\lambda(3n+1)=\lambda(3n+2)=+1$ holds with asymptotic probability $0$.  The claim then follows from the inclusion-exclusion principle.
\end{proof}

\begin{corollary}\label{312b}  Asymptotically almost surely, we have $\lambda(3n-1)=-\lambda(3n+1)$.
\end{corollary}

\begin{proof}  For $n$ even, this follows from Corollary~\ref{corp}, so suppose that $n=2N+1$.  But then since $\lambda(2m)=-\lambda(m)$, the claim $\lambda(3n-1)=-\lambda(3n+1)$ in this case is equivalent to $\lambda(3N+1)=-\lambda(3N+2)$, and the claim follows from Corollary \ref{312}.
\end{proof}

We can combine Corollary \ref{312}, \ref{312b} using the Dirichlet character $\chi_3$ to conclude that asymptotically almost surely, $\lambda \chi_3$ is constant on the set $\{ 3n-1, 3n+1, 3n+2\}$.  Shifting $n$ repeatedly by $1$ using Lemma \ref{linear}, we conclude that for any fixed $k \geq 1$, $\lambda \chi_3$ is a.a.s. constant on the set $\{3n-1, 3n+1, 3n+2, 3n+4, 3n+5, \dots, 3n+3k-1, 3n+3k+1, 3n+3k+2 \}$, and in particular
$$ |\sum_{j=0}^{3k} \lambda(3n+j) \chi_3(3n+j)| = 2k$$
a.a.s..  By Lemma \ref{linear} this implies that
$$ |\sum_{j=0}^{3k} \lambda(n+j) \chi_3(n+j)| \geq 2k - 6$$
(say) a.a.s..  But this contradicts Theorem \ref{mrt-3} if $k$ is large enough.  This (finally!) completes the proof of Theorem \ref{newmain} for the sign pattern ${\stackrel{\vee}{+}}{+}{+}$.  The case ${\stackrel{\vee}{-}}{-}{-}$ is proven similarly by reversing all the signs in the above argument.

\section{The ${\stackrel{\vee}{+}}{-}{+}$ case}

We now prove Theorem \ref{newmain} for the sign pattern ${\stackrel{\vee}{+}}{-}{+}$; the argument for ${\stackrel{\vee}{-}}{+}{-}$ is analogous and follows by reversing all the signs below.  Our arguments follow that of Hildebrand \cite{hil}, adapted to the notation of this paper.

For sake of contradiction, we assume that

\begin{hypothesis}\label{po2} We have $n \not \mapsto {\stackrel{\vee}{+}}{-}{+}$ a.a.s..
\end{hypothesis}

This leads to the following implications:

\begin{proposition}\label{clam}  Asymptotically almost surely, the following two claims hold:
\begin{itemize}
\item[(a)]  If $2n \mapsto {\stackrel{\vee}{+}}{+}$, then $3n \mapsto {\stackrel{\vee}{+}}{+}$.
\item[(b)]  If $2n \mapsto {+}{\stackrel{\vee}{+}}$, then $3n \mapsto {+}{\stackrel{\vee}{+}}$.
\end{itemize}
\end{proposition}

\begin{proof} We just prove (a), as (b) is analogous (reflecting all sign patterns around the positional marker $\vee$).  Suppose that $2n \mapsto {\stackrel{\vee}{+}}{+}$, then since $\lambda(3m)=-\lambda(m)$, we have
$$ 6n \mapsto {\stackrel{\vee}{-}}{\ast}{\ast}{-}.$$
Since $\lambda(2m)=-\lambda(m)$, our objective is to show that $\lambda(6n+2)=-1$ a.a.s..  Suppose instead that $\lambda(6n+2)=+1$, thus
$$ 6n \mapsto {\stackrel{\vee}{-}}{\ast}{+}{-}.$$
By Hypothesis \ref{po2} and Lemma \ref{linear}, we have
$$ 6n \not\mapsto {\stackrel{\vee}{\ast}}{\ast}{+}{-}{+}$$
and thus
$$ 6n \mapsto {\stackrel{\vee}{-}}{\ast}{+}{-}{-}$$
a.a.s..  But as $\lambda(2m)=-\lambda(m)$, this implies
$$ 3n \mapsto {\stackrel{\vee}{+}}{-}{+}$$
which contradicts Hypothesis \ref{po2} and Lemma \ref{linear} a.a.s..
\end{proof}

\begin{corollary}\label{clam-2}  Let $a, b$ be integers with $a < b$.  Asymptotically almost surely, if $\lambda(m)=+1$ for all $n+a \leq m \leq n+b$, then
$\lambda(m)=+1$ for all $\frac{3}{2} (n+a) \leq m \leq \frac{3}{2} (n+b)$.
\end{corollary}

\begin{proof} By the union bound and Lemma \ref{linear}, it suffices to verify this in the case $a=0, b=1$.  Proposition \ref{clam}(a) then handles the case when $n$ is even, and Proposition \ref{clam}(b) handles the case when $n$ is odd, and the claim then follows from Lemma \ref{linear}.
\end{proof}

\begin{corollary}\label{lo} For any natural number $k$, we have $\lambda(n)=\ldots=\lambda(n+k-1)=+1$ with positive asymptotic probability.
\end{corollary}

\begin{proof} We first establish the case $k=4$, which of course implies the $k=1,2,3$ cases as well.  Suppose for contradiction that the $k=4$ claim failed, thus
$$ n \not \mapsto {\stackrel{\vee}{+}}{+}{+}{+}$$
a.a.s..  In particular, from Lemma \ref{linear} one a.a.s. has
$$ 3n \not \mapsto {\stackrel{\vee}{+}}{+}{+}{+}.$$
By Proposition \ref{clam} in the contrapositive we then have
$$ 2n \not \mapsto {\stackrel{\vee}{+}}{+}{+}$$
a.a.s.; but by Hypothesis \ref{po2} and Lemma \ref{linear} we have
$$ 2n \not \mapsto {\stackrel{\vee}{+}}{-}{+}$$
and thus (since $\lambda(2m)=-\lambda(m)$)
$$ n \not \mapsto {\stackrel{\vee}{-}}{-}$$
a.a.s..  But this contradicts Theorem \ref{pap}.

Now assume inductively that the claim holds for some $k \geq 4$.  By partitioning according to the parity of $n$ and using Lemma \ref{linear}, we thus see that with positive asymptotic probability, either
$$ \lambda(2n) = \lambda(2n+1) = \dots = \lambda(2n+k-1) = +1$$
or
$$ \lambda(2n-1) = \lambda(2n) = \dots = \lambda(2n+k-2) = +1.$$
In the former case, we see from Corollary \ref{clam-2} (and the hypothesis $k \geq 4$, which implies that $k \leq 3\frac{k-1}{2}$) that
$$ \lambda(3n) = \dots = \lambda(3n+k) = +1,$$
and in the latter case we similarly have (since $k-1 \leq 3\frac{k-2}{2}$)
$$ \lambda(3n-1) = \lambda(3n) = \dots = \lambda(3n+k-1) = +1.$$
In either case we obtain the $k+1$ case of the claim from Lemma \ref{linear}.
\end{proof}

Next, for any fixed real number $a>0$, let $A_a \subset (a,+\infty)$ be the set
$$ A_a := \{ t \in (a,+\infty): \lambda(n)=+1 \hbox{ for all } t-a < n < t+a \}$$
and consider the quantity
$$ p_a := \operatorname{LIM}\left ( \frac{1}{\log i} \int_{x_i/i}^{x_i} 1_{A_a}(t) \frac{dt}{t} \right)_{i=1}^\infty.$$
From Corollary \ref{lo} applied with, say, $k=\lfloor a+10\rfloor$, and rounding to the nearest integer, we see that $p_a > 0$ for any fixed $a>0$.  On the other hand, from Theorem \ref{mrt} (and another rounding argument) we see that
\begin{equation}\label{pa-lim}
\lim_{a \to \infty} p_a = 0.
\end{equation}

We now claim that
\begin{equation}\label{pgrow}
 p_{\frac{3}{2}a + 10} \geq p_{a+10}
\end{equation}
for any $a \geq 0$; iterating this starting from (say) $a=1$, we see that $\limsup_{a \to \infty} p_a \geq p_{11} > 0$, contradicting \eqref{pa-lim}.

We now show \eqref{pgrow}.  Suppose that $t \in A_{a+10}$, thus
$$ \lambda(n)=+1 \hbox{ for all } t-a-10 < n < t+a+10.$$
Applying Corollary \ref{clam-2} (and rounding to the nearest integer), we then conclude that
$$ \lambda(n)=+1 \hbox{ for all } \frac{3}{2} t-\frac{3}{2} a-10 < n < \frac{3}{2} t+\frac{3}{2} a+10$$
for $t \in A_{a+10}$ outside of an exceptional set $E_a \subset (0,+\infty)$ which has zero asymptotic density in the sense that
$$  \operatorname{LIM} \left( \frac{1}{\log i} \int_{x_i/i}^{x_i} 1_{E_a}(t) \frac{dt}{t} \right)_{i=1}^\infty = 0.$$
We conclude that
$$ \operatorname{LIM} \left( \frac{1}{\log i} \int_{x_i/i}^{x_i} 1_{A_{\frac{3}{2}a+10}}\left(\frac{3}{2} t\right) \frac{dt}{t} \right)_{i=1}^\infty \geq p_a,$$
and hence by the change of variables $t' := \frac{3}{2} t$ 
$$ \operatorname{LIM} \left( \frac{1}{\log i} \int_{3x_i/2i}^{3x_i/2} 1_{A_{\frac{3}{2}a+10}}(t) \frac{dt}{t} \right)_{i=1}^\infty \geq p_a.$$
We may replace the limits of integration from $\int_{3x_i/2i}^{3x_i/2}$ to $\int_{x_i/i}^{x_i}$ incurring an error of $O\left( \frac{1}{\log i} \right)$ which vanishes in the limit, and \eqref{pgrow} follows.

\section{A random graph theory question}\label{Graph}

We now return to the proof of Theorem \ref{newmain2}.  In this section we will reduce this theorem to the task of establishing that a certain random graph is almost surely connected; this connectedness will be verified in the next section.  

Recall that the only remaining tasks are to show that $(\mu(n),\mu(n+1))$ attains each of the four values $(+1,+1), (+1,-1), (-1,+1), (-1,-1)$ with positive asymptotic probability.  From the asymptotic probabilities already computed in Section \ref{asym}, one can check that the four events
$$\mu(n)=+1, \mu^2(n+1)=1$$
$$\mu(n)=-1, \mu^2(n+1)=1$$
$$\mu^2(n)=1, \mu(n+1)=+1$$
$$\mu^2(n)=1, \mu(n+1)=-1$$
each occur with asymptotic probability $c/2$, where $c$ was defined in \eqref{cdef}.
In particular we see that $(\mu(n),\mu(n+1))$ takes the values $(+1,+1)$ and $(-1,-1)$ with equal asymptotic probability, and also takes the values $(+1,-1)$ and $(-1,+1)$ with equal asymptotic probability.  Thus, we only need to show that the values $(+1,+1)$ and $(+1,-1)$ are attained with positive asymptotic probability.

Suppose for sake of contradiction that there was a sign $\epsilon \in \{-1,+1\}$ such that $(\mu(n),\mu(n+1))$ avoided $(+1,\epsilon)$ a.a.s..  Then it also avoids $(-1,-\epsilon)$ a.a.s., and so we conclude that we have the implication
$$ \mu^2(n)=\mu^2(n+1)=1 \implies \mu(n) = - \epsilon \mu(n+1)$$
a.a.s..

We can eliminate the sign $-\epsilon$ as follows.  Define $\chi$ to be the completely multiplicative function such that $\chi(p) := +1$ for all $p>2$, and $\chi(2) := - \epsilon$.  Then $\chi(n) = - \epsilon \chi(n+1)$ for any $n$ for which $n,n+1$ are not divisible by $4$, and thus we have
$$ \mu^2(n)=\mu^2(n+1)=1 \implies \mu \chi(n) = \mu \chi(n+1)$$
a.a.s..  As $\mu \chi$ is a multiplicative function, we can use Lemma \ref{linear} and conclude more generally that
\begin{equation}\label{mund}
(\mu^2(n)=\mu^2(n+d)=1 \wedge d|n) \implies \mu \chi(n) = \mu \chi(n+d)
\end{equation}
a.a.s. for any fixed natural number $n$.

The strategy is now to use \eqref{mund} create large ``chains'' $n+a_1, n+a_2, \dots$ on which $\mu \chi$ is constant, and demonstrate that this is incompatible with Theorem \ref{mrt-3}.  It will be convenient to pass from the finitely additive world of asymptotic probability to the countably additive world of genuine probability.  Recall that the \emph{profinite integers} $\hat \Z$ are defined as the inverse limit of the cyclic groups $\Z/N\Z$; we embed the ordinary integers $\Z$ as  a subgroup of $\hat \Z$ in the usual fashion.  The group $\hat \Z$ is compact (in the profinite topology) and thus has a well-defined Haar probability measure, so we can meaningfully talk about a random profinite integer $\mathbf{n} \in \hat \Z$, whose reductions $\mathbf{n}\ (N)$ to any cyclic group $\Z/N\Z$ are uniformly distributed in that group.   (We will use boldface notation to indicate random variables that are generated from a random profinite integer.)  We say that a profinite integer $n$ is divisible by a natural number $N$ if the reduction of $n$ to $N$ vanishes.  The M\"obius function $\mu$ does not obviously extend to the profinite integers, but we can (by abuse of notation) define the quantity $\mu^2(n)$ for profinite $n$ to equal $1$ if $n$ is \emph{squarefree} in the sense that it is not divisible by $p^2$ for any prime $p$, and equal to $0$ otherwise.

We now construct a random graph $\mathbf{G} = (\mathbf{V}, \mathbf{E})$ as follows.   Let $\mathbf{n} \in \hat \Z$ be a random profinite integer, and define the random vertex set $\mathbf{V} \subset \Z$ to be the random set of integers defined by
$$ \mathbf{V} := \{ a \in \Z: \mu^2(\mathbf{n} + a) = 1 \}.$$
We then define $\mathbf{E}$ to be the set of pairs $\{a,b\}$ of distinct vertices $a,b$ in $\mathbf{V}$, such that $|b-a|$ is an odd prime number dividing $\mathbf{n}+a$ (or equivalently $\mathbf{n}+b$), and define the random graph $\mathbf{G}$ by setting $\mathbf{V}$ as the set of vertices and $\mathbf{E}$ as the set of edges.  (The restriction of $|b-a|$ to be an odd prime will be needed in order to keep the analysis of $\mathbf{G}$ tractable.)  Thus, for instance, the integers $2,5$ will be connected by an edge in $\mathbf{G}$ if $\mu^2({\mathbf n}+2)=\mu^2({\mathbf n}+5)=1$ and $3$ divides ${\mathbf n}+2$.

In the rest of the paper we will show

\begin{theorem}\label{surely}  The random graph $\mathbf{G}$ is almost surely connected.
\end{theorem}

Let us assume this theorem for now and finish establishing the contradiction required to conclude Theorem \ref{newmain2} and hence Theorem \ref{main-2}.  Let $a, b$ be integers.  By Theorem \ref{surely}, we see almost surely that if $\mu^2(\mathbf{n}+a)=\mu^2(\mathbf{n}+b)=1$, then there exists a finite path $a = c_1, c_2, \dots, c_k = b$ of distinct integers with $k \geq 1$ with the property that $\mu^2(\mathbf{n}+c_1)=\dots=\mu^2(\mathbf{n}+c_k)=1$ and such that $|c_{i+1}-c_i|$ divides $\mathbf{n}+c_i$ for each $i=1,\dots,k-1$.  (The $|c_{i+1}-c_i|$ are also odd primes, but we will discard this information as it will not be needed here.)
There are only countably many choices for $c_1,\dots,c_k$, so from countable additivity we see that for any $\eps > 0$ and $a,b \in \Z$ one can find a natural number $M$ (depending on $\eps,a,b$) with the property that with probability at least $1-\eps$, if $\mu^2(\mathbf{n}+a)=\mu^2(\mathbf{n}+b)=1$, then there exist distinct integers $a = c_1, c_2, \dots, c_k = b$ with $k \leq M$ and $|c_1|,\dots,|c_k| \leq M$, such that $\mu^2(\mathbf{n}+c_1)=\dots=\mu^2(\mathbf{n}+c_k)=1$ and $|c_{i+1}-c_i|$ divides $\mathbf{n}+c_i$ for each $i=1,\dots,k-1$.  

Note from the Chinese remainder theorem that for any fixed number of congruence conditions $n = a_i\ (q_i)$, the asymptotic probability that $n$ obeys these congruence conditions is equal to the probability that the random profinite integer $\mathbf{n}$ obeys the same congruence conditions.  Because of this, we can transfer the previous claim\footnote{Strictly speaking, conditions such as $\mu^2(\mathbf{n}+c_i)=1$ involve an infinite number of congruence conditions, but from the absolute convergence of $\sum_p \frac{1}{p^2}$ one can approximate these conditions by a finite number of congruence conditions at the cost of an arbitrarily small profinite probability or asymptotic probability; we leave the details to the reader.} from $\mathbf{n}$ to $n$.  That is to say, for any $\eps > 0$ and $a,b \in \Z$, there exists $M$ with the property that with asymptotic probability at least $1-\eps$, if $\mu^2(n+a)=\mu^2(n+b)=1$, then there exist distinct integers $a = c_1, c_2, \dots, c_k = b$ with $k \leq M$ and $|c_1|,\dots,|c_k| \leq M$, such that $\mu^2(n+c_1)=\dots=\mu^2(n+c_k)=1$ and $|c_{i+1}-c_i|$ divides $n+c_i$ for each $i=1,\dots,k-1$.  But from \eqref{mund}, this conclusion implies a.a.s. that $\mu \chi(n+c_{i+1}) = \mu \chi(n+c_i)$ for each $i=1,\dots,k-1$.  Chaining these equalities together, we conclude that with asymptotic probability at least $1-\eps$, we have
$$ \mu^2(n+a) = \mu^2(n+b) = 1  \implies \mu \chi(n+a) = \mu\chi(n+b).$$
Sending $\eps$ to zero, we conclude that this claim holds a.a.s. for any fixed choice of $a,b$.
Applying this for all $a,b \in \{1,\dots,h\}$, we conclude that for any fixed $h$, we have
$$ |\sum_{j=1}^h \mu \chi(n+j)| = \sum_{j=1}^h \mu^2(n+h) $$
a.a.s..  On the other hand, by using Theorem \ref{mrt-3} (treating the contribution of odd and even $n+j$ separately) we see that
$$ |\sum_{j=1}^h \mu \chi(n+j)| \leq \eps h $$
with asymptotic probability at least $1-\eps$, if $\eps>0$ and $h$ is sufficiently large depending on $\eps$.  Thus we see that the quantity $\sum_{j=1}^h \mu^2(n+h)$ has an asymptotic expectation of at most $2 \eps h$.  On the other hand, from linearity of expectation this asymptotic expectation is $\frac{h}{\zeta(2)}$, and one obtains a contradiction if $\eps$ is small enough and $h$ is large enough.

It remains to establish Theorem \ref{surely}.  In next section we will make several reductions, some of which are easy and some of which are more involved. In the following paragraph we give a sketch of all the reductions, forgetting about some additional technical conditions on the lengths of paths between elements, etc.. 

First, we will show that instead of showing that almost surely any vertices $a$ and $b$ are connected, it is enough to show that almost surely $0$ and $X$ are connected for any large enough odd $X$. Then, we will show that it is enough to consider the graph with a slight extension of the vertex set, where instead of requiring $\mu^2(\mathbf{n}+a) = 1$, one only requires that $\mathbf{n}+a$ is not divisible by $p^2$ for any prime $p \leq (\log X)^5$. Next, we will show that instead of connecting $0$ to all large enough vertices $X$, it is enough to connect it to at least $X^{9/10}$ even elements up to $X/10$, which means that it is enough to show that, for all $X' \asymp X$, $0$ is almost surely connected to at least one element in $[X'-X^{1/20}, X']$. Then, we will show that this follows if we can show that $0$ is connected to at least $\log^{10} X$ odd elements in $[0, X^{1/100}]$. Thus we will dramatically reduce the number of distinct vertices we need to connect $0$ to, and this last claim is shown through a very careful study of paths of length somewhat shorter than $\log \log x$ in Section~\ref{conclude-sec}.

\section{Initial reductions}\label{initial-sec}

For the remainder of this paper, we use the usual asymptotic notation of writing $X \ll Y$, $Y \gg X$, or $X = O(Y)$ to denote the estimate $|X| \leq CY$ for some absolute constant $C$, and write $X \asymp Y$ for $X \ll Y \ll X$.

We now turn to the formal proof of Theorem \ref{surely}.  We begin by deducing this theorem from the following claim.

\begin{proposition}\label{longpath}  Let $X$ be a sufficiently large odd natural number, which we view as an asymptotic parameter going to infinity.  Then, with probability $1-o(1)$, if $0$ and $X$ both lie in ${\mathbf V}$, then there is a path in ${\mathbf G}$ from $0$ to $X$.
\end{proposition}

Let us see how Proposition \ref{longpath} implies Theorem \ref{surely}.  Observe that the random graph ${\mathbf G}$ is stationary, in the sense that any translate of ${\mathbf G}$ has the same distribution as ${\mathbf G}$.  Thus it suffices to show that for any natural number $h$, one almost surely has $0$ and $h$ connected in ${\mathbf G}$ whenever $0, h$ both lie in ${\mathbf G}$.  Stationarity also shows that Proposition \ref{longpath} automatically extends to even $X$ as well as odd $X$, since one can write a large even number as the sum of two large odd ones.

Let $w$ be a natural number (larger than $h$), and let $W := \prod_{p\leq w} p$.   Note that if $0$ and $h$ both lie in ${\mathbf V}$, then $\mathbf{n}$ and $\mathbf{n}+h$ are not divisible by $p^2$ for any $p \leq w$. Conditioning on the event that $0$ and $h$ both lie in ${\mathbf V}$, we see that $\mathbf{n}+W^2$ is not divisible by $p^2$ for any $p \leq w$, and an application of the union bound shows that $\mathbf{n}+W^2$ is square-free with probability $1-O(1/w)$; that is to say $W^2 \in {\mathbf V}$ with probability $1-O(1/w)$.  Applying Proposition \ref{longpath} (and stationarity) we conclude that conditionally on $0$ and $h$ both lying in ${\mathbf V}$, we have a path in ${\mathbf G}$ from $0$ to $h$ through $W^2$ with probability $1- O(1/w) - o(1)$ as $W^2 \to \infty$.  Sending $w \to \infty$, we obtain Theorem \ref{surely}.

It remains to prove Proposition \ref{longpath}.  
Let $X$ be a sufficiently large odd natural number, viewed as an asymptotic parameter going to infinity.  It is convenient to work with a slight enlargement ${\mathbf G}_X$ of ${\mathbf G}$ with slightly more vertices.  Set
$$ w := \log^5 X $$
and define $\mu^2_X(n)$ for a profinite integer $n$ to equal one when $n$ is not divisible by $p^2$ for any $p \leq w$, and zero otherwise, and set ${\mathbf V}_X := \{ a \in \Z: \mu^2_X(\mathbf{n}+a) = 1 \}$; then ${\mathbf V}_X$ contains ${\mathbf V}$, and a vertex $a$ of ${\mathbf V}_X$ lies in ${\mathbf V}$ unless $p^2 | \mathbf{n}+a$ for some $p>w$.

Define $\mathbf{G}_X$ to be the random graph with vertex set ${\mathbf V}_X$, and two distinct elements $a,b$ of ${\mathbf V}_X$ connected by an edge if $|a-b|$ is an odd  prime dividing $\mathbf{n}+a$.  Thus $\mathbf{G}$ is the restriction of $\mathbf{G}_X$ to ${\mathbf V}$.  We will first show that it suffices to prove

\begin{proposition}\label{longpath-2}  Let $X$ be a sufficiently large odd natural number, which we view as an asymptotic parameter going to infinity.    Then, with probability $1-o(1)$, if $0$ and $X$ both lie in ${\mathbf V}_X$, then there is a path in ${\mathbf G}_X$ from $0$ to $X$ of length at most $10 \log^2 X$ and contained in $[-10X, 10X]$.
\end{proposition}

Let us see how Proposition \ref{longpath-2} implies Proposition \ref{longpath}.  The key point is that the restriction of ${\mathbf G}_X$ to $[-10X,10X]$ does not require knowledge of the entirety of the profinite random integer $\mathbf{n}$; only knowledge of the reductions $\mathbf{n} \hbox{ mod } p$ for $p \leq 20 X$  and $\mathbf{n} \hbox{ mod } p^2$ for $p \leq w$ is required.  The remaining components of $\mathbf{n}$ can then be used to restrict ${\mathbf G}_X$ to ${\mathbf G}$.  

Let us first show that, almost surely every number in $[\mathbf{n}-10X, \mathbf{n}+10X]$ has at most $\log^2 X$ distinct prime factors in the interval $[w, 20X]$. Notice that this property depends only on reductions modulo primes $p \leq 20X$. Write $W = \prod_{p \leq 20X} p = e^{20X(1+o(1))}$, and then it is enough to show that the number of integers in $[W, 2W)$ that have more than $\log^2 X$ distinct prime factors in $[w, 20X]$ is $o(W/X)$. Such numbers have $20X$-smooth part at least $w^{\log^2 X}$, and we get that the number of them is at most
\[
\sum_{\substack{W \leq mn \leq 2W \\ p \mid m \implies p > 20X \\ p \mid n \implies p \leq 20X \\ n \geq w^{\log^2 X}}} 1 \leq \sum_{\substack{m \leq 2W \\ p \mid m \implies p > 20X}} \sum_{\substack{w^{\log^2 X} \leq n \leq 2W/m \\ p \mid n \implies p \leq 20X}} 1 \ll \sum_{\substack{m \leq 2W}} \frac{2W}{m} (\log X)^{-c\log X \log \log X} \ll \frac{W}{X^{100}}
\]
for an absolute constant $c>0$, by the standard estimate (see e.g. \cite{bruijn}) that the number of $y$-smooth numbers up to $x$ is $u^{-(1+o(1))u}$ where $u = \log x/\log y \leq y^{1-\varepsilon}$ (in our case, to get an upper bound, we can take, for every $m$, $u = \log(w^{\log^2 X}) / \log (20X) \asymp \log X \log \log X$).

Now suppose we condition the reductions $(\mathbf{n} \hbox{ mod } p)_{p \leq 20 X}$ and $(\mathbf{n} \hbox{ mod } p^2)_{p \leq w}$ to be a value for which there is a path in ${\mathbf G}_X$ from $0$ to $X$ of length at most $10 \log^2 X$ contained in $[-10X, 10 X]$, and for which every number in $[\mathbf{n}-10X, \mathbf{n}+10X]$ has at most $\log^2 X$ prime factors in $[w, 20X]$.  After this conditioning, the residue classes $\mathbf{n} \hbox{ mod } p^2$ for $w < p \leq 20X$ are restricted to a single coset of $\Z/p\Z$ in $\Z/p^2\Z$, but are uniformly distributed in that coset, whereas the residue classes $\mathbf{n} \hbox{ mod } p^2$ for $p > 20X$ are uniformly distributed on all of $\Z/p^2 \Z$. Also, the $\mathbf{n} \hbox{ mod } p^2$ are independent in $p$ across all primes $p$.  Let $0 = a_1, a_2, \dots, a_k = X$ be a path in ${\mathbf G}_X$ from $0$ to $X$ of length $k \leq 10 \log^2 X$ contained in $[-10X, 10X]$. Then, for each $j = 1, \dotsc, k$, the number of primes $p \in [w, 20X]$ such that $p$ divides $\mathbf{n} + a_k$ is at most $\log^2 X$. Hence the probability that all of the $\mathbf{n}+a_i$ are not divisible by $p^2$ for any $p > w$ (which implies that this path lies in ${\mathbf G}$ and not just in ${\mathbf G}_X$) is at least
$$ \left(1 - \frac{k}{w}\right)^{\log^2 X} \times \prod_{p > 20X} \left(1-\frac{k}{p^2}\right).$$
From the bounds on $w,k$ we see that this expression is $1-o(1)$.
This then gives Proposition \ref{longpath} from Proposition \ref{longpath-2}.

We will want to substantially reduce the number of vertices to which we need to connect $0$. To do this, we will use the following lemma several times.
\begin{lemma}
\label{le:ABconnect}
Let $X$ be large and fix the residue classes $\mathbf{n} \hbox{ mod } p$ for $p \leq X$ and $\mathbf{n} \hbox{ mod } p^2$ for $p \leq w$. Let $A$ and $B$ be respectively subsets of odd and even integers in $[0, X] \cap {\mathbf V}_X$ such that $|A||B| \gg X \log^{10} X$. Then, with conditional probability $1-O(\log^{-2} X)$, there is a path of length $3$ in ${\mathbf V}_X$ contained in $[-8 X, 8 X]$ connecting an element $a \in A$ with an element $b \in B$.
\end{lemma}

\begin{proof}
Let $S$ denote the set of quadruples $(a,b,p_1,p_2,p_3)$ obeying the following constraints:
\begin{itemize}
\item We have $a \in A, b \in B$, and $p_1,p_2,p_3$ are primes in $(X, 3X]$, $(5X, 7X]$, $(3X, 5X]$ respectively obeying the equation
$$ b = a - p_1 + p_2 - p_3.$$
\item $a-p_1$ and $a-p_1+p_2$ lie in ${\mathbf V}_X$, that is to say $\mathbf{n}+a-p_1, \mathbf{n}+a-p_1+p_2$ are not divisible by $p^2$ for any $p \leq w$.
\end{itemize}

Note that the set $S$ is deterministic due to our conditioning.  A routine application of the circle method (see Proposition \ref{prop:Vinogradov} below) shows that each pair $(a,b) \in A \times B$ contributes $\asymp \frac{X^2}{\log^3 X}$ quintuples to $S$, and so
\begin{equation}\label{slower}
 |S| \asymp |A| |B| X^2 / \log^3 X.
\end{equation}
Each quintuple $(a,b,p_1,p_2,p_3)$ will yield a path $a, a-p_1, a-p_1+p_2, a-p_1+p_2-p_3=b$ of the desired form provided that one has the divisibility conditions
$$ p_1 | \mathbf{n} + a; \quad p_2 | \mathbf{n} + a - p_1; \quad p_3 | \mathbf{n} + a - p_1 + p_2.$$
Let $E_{(a,b,p_1,p_2,p_3)}$ be the event that these three divisibility conditions occur; thus it suffices to show that
$$ \P\left( \bigvee_{s \in S} E_s \right) = 1-O((\log X)^{-2}).$$
We use the second moment method. By the Cauchy-Schwarz inequality, 
\[
\begin{split}
\P( \bigvee_{s \in S} E_s ) \geq \frac{( \E \sum_{s \in S} 1_{E_s})^2}{\E (\sum_{s \in S} 1_{E_s})^2} = \frac{\sum_{s,s' \in S} \P(E_s) \P(E_{s'})}{\sum_{s,s' \in S} \P( E_s \cap E_{s'} )}.
\end{split}
\]
Hence it suffices to show that
\begin{equation}\label{duplex}
\sum_{s,s' \in S:  \P( E_s \cap E_{s'} ) > \P(E_s) \P( E_{s'} ) } \P( E_s \cap E_{s'} )  \ll \frac{1}{\log^2 X} \sum_{s,s' \in S} \P(E_s) \P(E_{s'}).
\end{equation}

On the one hand, from the Chinese remainder theorem we have
$$ \P( E_s ) = \frac{1}{p_1 p_2 p_3} \asymp \frac{1}{X^3}$$
and so by \eqref{slower}
\begin{equation}
\label{eq:PPasymp} 
\sum_{s,s' \in S} \P(E_s) \P(E_{s'}) \asymp \frac{|A|^2 |B|^2}{X^2 \log^6 X}.
\end{equation}
On the other hand, if $s=(a,b,p_1,p_2,p_3)$ and $s'=(a',b',p'_1,p'_2,p'_3)$ lie in $S$, then $\P( E_s \cap E_{s'} )$ is equal to either $\P(E_s) \P(E_{s'})$ or $0$ unless at least one of the following three situations occur:
\begin{itemize}
\item[(i)] $p_1 = p'_1$ and $a = a'$.
\item[(ii)] $p_2 = p'_2$ and $a-p_1 = a'-p'_1$.
\item[(iii)] $p_3=p'_3$ and $b = b'$.
\end{itemize}
(Here we are using the fact that $p_1,p_2,p_3$ lie in disjoint intervals, and $p_1$ is larger than the diameter of the interval in which $a$ ranges, $p_2$ is larger than the diameter of the interval in which $a-p_1$ ranges and $p_3$ is larger than the diameter of the interval in which $b$ ranges.) Furthermore, in these exceptional
cases, $\P( E_s \cap E_{s'} )$ may be bounded by $O( \frac{1}{X^{6-j}} )$, where $j=1,2,3$ is the number of situations (i), (ii), (iii) that are occurring simultaneously.  

Meanwhile, when $j$ of the situations occur simultaneously, a simple degree of freedom counting gives $O(|A||B| (X/\log X)^2)$ choices for $s$ and then $O((X/\log X)^{3-j})$ choices for $s'$ (one can choose the primes $p'_j$ not fixed by the first conditions in (i)--(iii) freely, and then everything else is fixed by the conditions and definition of $S$). Thus the left-hand side of \eqref{duplex} is bounded by
$$O\left( \sum_{j=1}^3 |A||B| \left(\frac{X}{\log X}\right)^{5-j} \cdot \frac{1}{X^{6-j}}\right) = O\left(\frac{|A||B|}{X (\log X)^2}\right) $$
and~\eqref{duplex} follows from \eqref{eq:PPasymp} since $|A||B| \geq X(\log X)^{10}$.
\end{proof}

Now we use the previous lemma to make a reduction, in which we content ourselves with connecting $0$ to many vertices in an interval $[0,X/10]$, rather than trying to reach a specific vertex $X$.  Namely, we reduce Proposition \ref{longpath-2} to

\begin{proposition}\label{longpath-3}  Let $X$ be a sufficiently large odd natural number, which we view as an asymptotic parameter going to infinity.    Then, with probability $1-o(1)$, if $0$ lies in ${\mathbf V}_X$, then there are at least $X^{9/10}$ even elements of $[0,X/10]$ which are connected to $0$ in ${\mathbf G}_X$ by a path of length at most $2\log^2 X$ contained in $[0,X/10]$.
\end{proposition}

We now explain why Proposition \ref{longpath-3} implies Proposition \ref{longpath-2}.  For any integer $a$, let $E_a$ be the event that $a$ lies in ${\mathbf V}_X$, 
and there are at least $X^{9/10}$ elements of $[a,a+X/10]$ of the same parity as $a$ which are connected to $a$ in ${\mathbf G}_X$ by a path of length at most $2\log^2 X$ contained in $[a,a+X/10]$.  By stationarity and Proposition \ref{longpath-3}, we see that for all $a$, one has with probability $1-o(1)$ that $E_a$ holds whenever $a \in {\mathbf V}_X$.  In particular, with probability $1-o(1)$, one has $E_0 \cap E_X$ holding whenever $0,X \in {\mathbf V}_X$.  

Now suppose that $E_0 \cap E_X$ holds.  This event only depends on the residue classes $\mathbf{n} \hbox{ mod } p$ for $p \leq X/10$ and $\mathbf{n} \hbox{ mod } p^2$ for $p \leq w$, so we now condition these residue classes to be deterministic.  We now have deterministic subsets $A, B$ of $[0,X/10]$ and $[X,11X/10]$ respectively of cardinality at least $X^{9/10}$ each, such that any element of $A$ (resp. $B$) is connected to $0$ (resp. $X$) by a path in ${\mathbf G}_X$ of length at most $2 \log^2 X$ contained in $[-10X,10X]$.  Furthermore, $A$ consists entirely of even numbers, and $B$ consists entirely of odd numbers. The claim now follows from Lemma \ref{le:ABconnect}.

Next, we reduce Proposition \ref{longpath-3} to the following variant, in which we connect $0$ to one element in a short interval, as opposed to many elements in a long interval:

\begin{proposition}\label{longpath-4}  Let $X$ be a sufficiently large odd natural number, which we view as an asymptotic parameter going to infinity.    Let $X'$ be an element of the interval $[X/40, X/20]$, chosen uniformly at random.  Then, with probability $1-o(1)$, if $0$ lies in ${\mathbf V}_X$, then there is an element of $[X'-X^{1/20}, X']$ that is connected to $0$ in ${\mathbf G}_X$ by a path of length at most $2\log^2 X$ contained in $[0,X/10]$.
\end{proposition}

Indeed, assuming Proposition \ref{longpath-4}, observe that if $0$ lies in ${\mathbf V}_X$ and $X'$ is chosen uniformly at random from $[X/40, X/20]$, then the expected number of elements in $[X'-X^{1/20}, X']$ that are connected to $0$ in the indicated fashion is at least $1-o(1)$.  From linearity of expectation, this implies that the number of elements of $[X/40-X^{1/20}, X/20]$ that are connected to $0$ in the indicated fashion is $\gg X^{1-1/20}$, and Proposition \ref{longpath-3} follows.

Finally, we reduce to a weaker version of Proposition \ref{longpath-3}, in which $0$ connects to far fewer elements in (a somewhat narrower) interval:

\begin{proposition}\label{longpath-5}  Let $X$ be a sufficiently large odd natural number, which we view as an asymptotic parameter going to infinity.    Then, with probability $1-o(1)$, if $0$ lies in ${\mathbf V}_X$, then there are at least $\log^{10} X$ odd elements of $[0,X^{1/100}]$ which are connected to $0$ in ${\mathbf G}_X$ by a path of length at most $\log X$ contained in $[0,X^{1/100}]$.  
\end{proposition}

Let us now see how Proposition \ref{longpath-5} implies Proposition \ref{longpath-4}.  Call an integer $a$ \emph{good} if $a \in {\mathbf V}_X$ is even, and $a$ is connected to at least $\log^{10} X$ odd elements of $[a, a+X^{1/100}]$ in ${\mathbf G}_X$ by a path of length at most $\log X$ contained in $[a,a+X^{1/100}]$.  From Proposition \ref{longpath-5} and stationarity, we see that each even element $a$ of ${\mathbf V}_X$ is good with probability $1-o(1)$.  From linearity of expectation, we thus see with probability $1-o(1)$ that all but $o(X)$ of the even elements of ${\mathbf V}_X \cap [-X,X]$ are good.

Note that the property of an integer being good is only dependent on the values of $\mathbf{n} \hbox{ mod } p^2$ for $p \leq w$ and $\mathbf{n} \hbox{ mod } p$ for $w < p \leq X^{1/100}$.  We now condition on these values to be fixed, in such a fashion that all but $o(X)$ of the even elements of ${\mathbf V}_X \cap [-X,X]$ are good; now the property of being good is deterministic, as is the vertex set ${\mathbf V}_X$.  We now use the following weak version of the Hardy-Littlewood inequality.
\begin{proposition} \label{prop:maximal}
Let $a_n$ be a sequence of non-negative real numbers supported in $[-X, X]$. Then,
\begin{equation} \label{equ:statement}
\frac{1}{X} \sum_{|n| \leq X} \Big ( \sup_{r \geq 1} \frac{1}{r} \sum_{|n - m| \leq r} a_m \Big )
\leq C \Big ( \frac{1}{X} \sum_{|n| \leq X} a_n^2 \Big )^{1/2}
\end{equation}
for some absolute constant $C > 0$. 
\end{proposition}
\begin{proof}  The sequence $\left(\sup_{r \geq 1} \frac{1}{r} \sum_{|n - m| \leq r} a_m\right)_{n \in \Z}$ is the (discrete) Hardy-Littlewood maximal operator applied to this sequence $(a_n)_{n \in \Z}$.  As the Hardy-Littlewood maximal operator is bounded in $\ell^2(\Z)$ (see e.g. \cite{stein:large}), we have
$$ \left(\sum_n \Big ( \sup_{r \geq 1} \frac{1}{r} \sum_{|n - m| \leq r} a_m \Big)^2\right)^{1/2}
\leq C \Big ( \sum_{|n| \leq X} a_n^2 \Big )^{1/2}$$
for some $C>0$, and the claim then follows from the Cauchy-Schwarz inequality.
\end{proof}

Choosing $a_m$ so that $a_m = 1$ if $m \in V_{\mathbf{X}} \cap [-X, X]$ is not good, and $a_m = 0$ otherwise, 
the proposition implies that if $X'$ is chosen uniformly at random from $[X/40, X/20]$, then with probability $1-o(1)$, the quantity $X'$ is ``excellent'' in the sense that for any $1 \leq r \leq X$, all but at most $o(r)$ of the even elements of ${\mathbf V}_X \cap [X'-r, X'+r]$ are good.  Note that because of our conditioning, the property of being excellent is deterministic.

Now let $X' \in [X/40, X/20]$ be an excellent number.  We will show that with probability $1-o(1)$, if $0 \in {\mathbf G}_X$, then there is an element of $[X'-X^{1/20}, X']$ that is connected to $0$ in ${\mathbf G}_X$ by a path of length at most $2\log^2 X$ contained in $[0,X/10]$; this clearly suffices to give Proposition \ref{longpath-4}.

We introduce the scales $R_j := 100^{1-j} X'$ for $j=1,\dots,J$, where $J$ is the first number for which $R_J \leq X^{1/40}$, thus $J \leq \log X$.  Introduce the intervals $I_j := [X'-R_j, X'- 0.99 R_j]$, thus the $I_j$ are disjoint, with $I_1$ containing $0$ and $I_J$ contained in $[X'-X^{1/20}, X']$.  
We will ``hop'' from $0$ to $I_J$ by a path passing through each of the $I_j$ in turn.  More precisely, let $A_j$ denote the even elements of ${\mathbf V}_X \cap I_j$ that are good; with all of our conditioning, this is a deterministic set.  A routine sieve shows that there are $\gg R_j$ even elements of ${\mathbf V}_X \cap I_j$, and as $X'$ is excellent, all but at most $o(R_j)$ of these elements are good.  We conclude that
$$ |A_j| \gg R_j.$$

We shall shortly establish the following lemma allowing one to hop from $A_j$ to $A_{j+1}$:

\begin{lemma}\label{load}  Let $1 \leq j < J$, and let $a_j \in A_j$, which is allowed to be a random variable depending on the values of $\mathbf{n} \hbox{ mod } p$ for $p \leq R_j/10$ but not on the reductions for higher $p$.  Then, with probability $1 - O( \frac{1}{\log^2 X} )$, there is a path of length at most $\log X + 3$ in ${\mathbf G}_X$ in $[0,X/10]$ connecting $a_j$ to an element $a_{j+1}$ of $A_{j+1}$.  Furthermore, $a_{j+1}$ depends only on the values of $\mathbf{n} \hbox{ mod } p$ for $p \leq R_{j+1}/10$.
\end{lemma}

Iterating this lemma with the union bound, starting from $a_1 = 0$, we conclude with probability $1 - O( \frac{J}{\log^2 X} ) = 1-o(1)$ that there is a path of length at most $J (\log X + 3) \leq 2 \log^2 X$ in ${\mathbf G}_X$ in $[0,X/10]$ connecting $0$ to an element of $A_J \subset [X'-X^{1/20},X']$, giving Proposition \ref{longpath-4}.  Thus it suffices to prove Lemma \ref{load}.  

We do this by an argument similar to that used to obtain Proposition \ref{longpath-3} from Proposition \ref{longpath-2}.  Condition on the values of $\mathbf{n} \hbox{ mod } p$ for $p \leq R_j/10$, so that $a_j$ is now deterministic.  By definition of $A_j$, $a_j$ is connected to a (deterministic) set $A$ by paths of length at most $\log X$ in ${\mathbf G}_X$ in $[0,X/10]$, where $A \subset [a_j, a_j + X^{1/100}]$ is a collection of odd numbers of cardinality
$$ |A| \geq \log^{10} X.$$

Since $ |A| |A_{j+1}| \gg R_j \log^{10} X$, Lemma~\ref{le:ABconnect} implies that there is a path of length $3$ connecting some $a \in A$ to some $b \in A_{j+1}$, and the claim follows.

It remains to prove Proposition \ref{longpath-5}.  This will be done in the next section.

\section{Conclusion of the argument}\label{conclude-sec}

We now prove Proposition \ref{longpath-5}. Condition the residue classes $\mathbf{n} \hbox{ mod } p^2$ for $p \leq w$ to be fixed, which makes the vertex set $\mathbf{V}_X$ deterministic, while keeping the $\mathbf{n} \hbox{ mod } p$ for $p>w$ uniformly and independently distributed on $\Z/p\Z$.

Set $k$ to be the odd number
\begin{equation}\label{k-def}
 k := 100 \left\lfloor \frac{\log\log X}{\sqrt{\log\log\log X}} \right\rfloor + 1;
\end{equation}
the reason for this somewhat strange choice is that $(\log\log X)^k$ will be significantly larger than $\log^{10} X$, while $k$ remains significantly smaller than $\log\log X$.  Let $\Gamma$ be the set of paths $\gamma$ in ${\mathbf V}_X$ of the form
\begin{equation}\label{ppk}
 0, p_1, p_1+p_2, \dots, p_1+\dots+p_k 
\end{equation}
where $p_1,\dots,p_k$ are distinct primes in the interval $I := [\exp(\sqrt{\log X}), X^{1/200}]$.
We write $\gamma(k) := p_1 + \dots + p_k$ for the endpoint of such a path, which is automatically odd since $k$ and the $p_1,\dots,p_k$ are odd, and by abuse of notation write $\gamma \subset {\mathbf G}_X$ if the path $\gamma$ lies in ${\mathbf G}_X$, or equivalently that
$$ p_i | {\mathbf n} + p_1 + \dots + p_{i-1} $$
for all $i=1,\dots,k$.

It will be technically convenient to weight the paths $\gamma$ in $\Gamma$.  For each path $\gamma$ of the form \eqref{ppk}, define the weight $w_\gamma > 0$ by the formula
$$ w_\gamma := \prod_{i=1}^{k} \frac{1}{\sum_{p \in I: p_1+\dots+p_{i-1}+p \in {\mathbf V}_X} \frac{1}{p}}.$$

\begin{lemma}\label{chic}  For every $\gamma \in \Gamma$, one has
$$ \sum_{p \in I: p_1+\dots+p_{i-1}+p \in {\mathbf V}_X} \frac{1}{p} \asymp \log\log X $$
for all $i=1,\dots,k$, so that
\begin{equation}\label{wg}
 w_\gamma = \exp( O(k) ) / (\log\log X)^k
\end{equation}
\end{lemma}

\begin{proof} Since $p \leq X^{1/200}$ the upper bound
$$                                                                                         
\sum_{p \in I: p_1 + p_2 + \ldots + p_{i-1} + p \in \mathbf{V}_X} \frac{1}{p} \leq \sum_{p \in I} \frac{1}{p} =                        
\frac{1}{2}\log\log X + O(1)                                                                                                                      
$$
is clear. For the lower bound note that since the reduction $\mathbf{n} \pmod{p^2}$ is fixed for $p \leq w$, there is a positive integer $A \leq X + \prod_{p \leq w} p^2$ such that the condition $p_1 + \ldots + p_{i-1} + p \in \mathbf{V}_X$ is equivalent to the condition $q^2 \nmid p + A$ for all $q \leq w$.
Pick a large constant $C$, let $\mathcal{P} = \prod_{p \leq C} p^2$. Note that by the Chinese Remainder Theorem
there are $\ell = \prod_{p \leq C} (\varphi(p^2) - 1)$ residues classes $a_1, \ldots, a_{\ell} \pmod{\mathcal{P}}$
such that if $p$ belongs to one of them then $q^2 \nmid p + A$ for all $q \leq C$. Therefore,
$$                                                                                                                                     
\mathbf{1}_{p_1 + p_2 + \ldots + p_{i-1} + p \in \mathbf{V}_X} \geq                                                                    
\sum_{i = 1}^{\ell} \mathbf{1}_{p \equiv a_{i} \pmod{\mathcal{P}}}                                                                     
- \sum_{C < q \leq w} \mathbf{1}_{p \equiv - A \pmod{q^2}}                                                                             
$$
Summing this with a weight of $1/p$ and using the prime number theorem in arithmetic progressions we conclude that
$$                                                                                                                                     
\sum_{p\in I: p_1 + \ldots + p_{i-1} + p \in \mathbf{V}_X} \frac{1}{p} \geq                                                           
\Big ( \prod_{p \leq C} \frac{\varphi(p^2) - 1}{\varphi(p^2)} - \frac{B}{C} \Big ) \sum_{p \in I} \frac{1}{p}                          
$$
for some absolute constant $B > 0$. Therefore if $C$ is choosen large enough then we obtain the desired lower bound.
\end{proof}

Let us first show that Proposition~\ref{longpath-5} follows once we have shown the three estimates
\begin{equation}\label{wsum}
\sum_{\gamma \in \Gamma} w_\gamma \P( \gamma \subset {\mathbf G}_X ) = 1 + o(1),
\end{equation}
\begin{equation}\label{meander}
 \sum_{\gamma,\gamma' \in \Gamma} w_\gamma w_{\gamma'} \P( \gamma, \gamma' \subset {\mathbf G}_X ) \leq 1 + o(1),
\end{equation}
and
\begin{equation}\label{meander-2}
 \sum_{\gamma,\gamma' \in \Gamma: \gamma(k)=\gamma'(k)} w_\gamma w_{\gamma'} \P( \gamma, \gamma' \subset {\mathbf G}_X ) \ll \log^{-100} X.
\end{equation}
Indeed, from \eqref{wsum}, \eqref{meander}, and Chebyshev's inequality we have
$$ \sum_{\gamma \in \Gamma} w_\gamma 1_{\gamma \subset {\mathbf G}_X} = 1+o(1)$$
with  probability $1-o(1)$, while \eqref{meander-2} and Markov's inequality gives
$$ \sum_{\gamma,\gamma' \in \Gamma: \gamma(k)=\gamma'(k)} w_\gamma w_{\gamma'} 1_{\gamma,\gamma' \subset {\mathbf G}_X} \ll \log^{-99} X $$
with probability $1-o(1)$. From Cauchy-Schwarz we have
\[
\begin{split}
\left(\sum_{\gamma \in \Gamma} w_\gamma 1_{\gamma \subset {\mathbf G}_X}\right)^2 &= \left(\sum_m \sum_{\gamma \in \Gamma \colon \gamma(k) = m} w_\gamma 1_{\gamma \subset {\mathbf G}_X}\right)^2 \\
&\leq |\{\gamma(k) \colon \gamma \in \Gamma, \gamma \subset {\mathbf G}_X\}| \cdot \sum_{\substack{\gamma, \gamma' \in \Gamma \\ \gamma(k) = \gamma'(k)}} w_\gamma w_{\gamma'} 1_{\gamma,\gamma' \subset {\mathbf G}_X}.
\end{split}
\]
Hence once we have shown the bounds \eqref{wsum}, \eqref{meander}, \eqref{meander-2}, we get
\[
|\{\gamma(k) \colon \gamma \in \Gamma, \gamma \subset {\mathbf G}_X\}| \gg (\log X)^{99}
\]
with probability $1-o(1)$ and Proposition \ref{longpath-5} follows.

We begin with \eqref{wsum}.  From the Chinese remainder theorem we have
$$ \P( \gamma \subset {\mathbf G}_X ) = \frac{1}{p_1 \dots p_k} $$
for a path $\gamma$ of the form \eqref{ppk}, since the $p_1,\dots,p_k$ were assumed to be distinct; thus the left-hand side of \eqref{wsum} becomes
\begin{equation}\label{sand}
\sum_{\gamma \in \Gamma} \frac{w_\gamma}{p_1 \dots p_k}.
\end{equation}
We can interpret this expression probabilistically as follows.  Consider a random path $0, p_1, p_1+p_2,\dots,p_1+\dots+p_k$, constructed iteratively by requiring that whenever $1 \leq i \leq k$ and $p_1,\dots,p_{i-1}$ have already been chosen, then $p_i \in I$ is chosen with probability
$$ \frac{1/p_i}{\sum_{p \in I: p_1+\dots+p_{i-1}+p \in {\mathbf V}_X} 1/p} $$
if $p_1+\dots+p_i \in {\mathbf V}_X$, and chosen with probability zero otherwise.  Then the quantity \eqref{sand} is nothing more than the probability that this random path actually lies in $\Gamma$.  This gives the upper bound for \eqref{wsum} automatically.  For the lower bound, observe that the only way the path $0,p_1,\dots,p_1+\dots+p_k$ could fail to lie in $\Gamma$ is if there is a collision $p_i=p_j$ for some $1 \leq i < j \leq k$.  But if $1 \leq i < j$, then after fixing $i,j$ and $p_1,\dots,p_{j-1}$ we see from Lemma \ref{chic} that the probability of the event $p_i = p_j$ occurring is $\ll \frac{1}{p_i \log\log X} \ll \exp( - \sqrt{\log X} )$, for a total failure probability of $\ll k^2 \exp( - \sqrt{\log X} ) = o(1)$.  This proves \eqref{wsum}.

Now we prove \eqref{meander}. By~\eqref{wsum}, it suffices to show that
$$
 \sum_{\gamma,\gamma' \in \Gamma: \P( \gamma,\gamma' \subset {\mathbf G}_X) > \P( \gamma \subset {\mathbf G}_X) \P( \gamma' \subset {\mathbf G}_X)
} w_\gamma w_{\gamma'} \P( \gamma, \gamma' \subset {\mathbf G}_X ) \leq o(1).$$
Let $\gamma = 0, p_1,\dots,p_1+\dots+p_k$ and $\gamma' = 0,p'_1,\dots,p'_1+\dotsb + p'_k$ be two paths in $\Gamma$.  The quantity
$\P( \gamma,\gamma' \subset {\mathbf G}_X)$ is usually equal to either $\P( \gamma \subset {\mathbf G}_X) \P( \gamma' \subset {\mathbf G}_X)$ or zero.  The only exceptions occur if we have at least one collision of the form $p_i = p'_j$ for some $1 \leq i,j \leq k$.  Furthermore, if such a collision occurs, the quantity $p_1+\dots+p_{i-1} - p'_1 - \dots - p'_{j-1}$ must be divisible by $p_i$.  If there are exactly $r$ collisions $p_{i_l} = p'_{j_l}$ for some pairs $(i_1,j_1),\dots,(i_r,j_r) \in \{1,\dots,k\}^2$, then we have
$$ \P( \gamma, \gamma' \subset {\mathbf G}_X ) \leq \frac{ p_{i_1} \dots p_{i_r} }{ p_1 \dots p_k p'_1 \dots p'_k }.$$
It thus suffices to show that
\begin{equation}\label{slo}
\begin{split}
& \sum_{r=1}^k \sum_{(i_1,j_1),\dots,(i_r,j_r)} \sum_{\gamma,\gamma' \in \Gamma} w_\gamma w_{\gamma'} \prod_{l=1}^r 1_{p_{i_l}=p'_{j_l}} 1_{p_{i_l}|p_1+\dots+p_{i_l-1} - p'_1 - \dots - p'_{j_l-1}} \\
&\quad \frac{ p_{i_1} \dots p_{i_r} }{ p_1 \dots p_k p'_1 \dots p'_k } = o(1),
\end{split}
\end{equation}
where the second sum is over distinct pairs $(i_1,j_1),\dots,(i_r,j_r)$ in $\{1,\dots,k\}^2$, with the $i_1,\dots,i_r$ and the $j_1,\dots,j_r$ distinct, with the ordering $i_1 < \dots < i_r$ (to avoid duplicates).

Consider first the contribution of the case where $i_l=l$ for $l=1,\dots,r$.
In this case we will omit the condition $1_{p_{i_l}|p_1+\dots+p_{i_l-1} - p'_1 - \dots - p'_{j_l-1}}$, which in principle gives a large reduction to the size of the expression in \eqref{slo}, but is difficult to analyse. We also drop
the condition that $\gamma, \gamma' \in \Gamma$ (thus allowing duplicates among $p_i$ and among $p_i'$). We can thus bound the left-hand side of this contribution to \eqref{slo} by
\begin{equation}\label{soar}
\sum_{r=1}^k \sum_{(1,j_1),\dots,(r,j_r)} \E \prod_{l=1}^r (p_{l} 1_{p_{l}=p'_{j_l}})
\end{equation}
the paths $\gamma = (0, p_1,\dots,p_1+\dots+p_k)$ and $\gamma' = (0,p'_1,\dots,p'_1+\dots+p'_k)$ are selected randomly as in the proof of \eqref{wsum}.  Observe that if one conditions $p'_1,\dots,p'_k$ and $p_1,\dots,p_{l-1}$ to be fixed, then $p_l 1_{p_l = p'_{j_l}}$ has conditional expectation
$$
\frac{1}{\sum_{p \in I: p_1+\dots+p_{i-1}+p \in {\mathbf V}_X} \frac{1}{p}} \ll \frac{1}{\log \log X}$$
by Lemma \ref{chic}.  Thus by multiplying together the conditional expectations, we can bound \eqref{soar} by
$$
\sum_{r=1}^k \sum_{(1,j_1),\dots,(r,j_r)} O\left( \frac{1}{\log\log X} \right)^r.$$
Because we have constrained $i_l=l$ for $l=1,\dots,r$, there are only $k^r$ choices for the $(i_1,j_1),\dots,(i_r,j_r)$, so we can bound the total contribution to \eqref{soar} or \eqref{slo} by
$$
\sum_{r=1}^k O\left( \frac{k}{\log\log X} \right)^r$$
which is (barely) of the form $o(1)$ thanks to \eqref{k-def}.

Now we consider the contribution of the case where $i_l \neq l$ for at least one $1 \leq l \leq r$; in particular, there exists $1 \leq l_0 \leq r$ such that $i_{l_0} > i_{l_0-1} + 1$ (with the convention that $i_0=0$).   We temporarily fix $r$, $(i_1,j_1),\dots,(i_r,j_r)$ and $l_0$ and consider the corresponding component of 
\begin{equation}\label{slo2}
\sum_{\gamma,\gamma' \in \Gamma} w_\gamma w_{\gamma'} \prod_{l=1}^r 1_{p_{i_l}=p'_{j_l}} 1_{p_{i_l}|p_1+\dots+p_{i_l-1} - p'_1 - \dots - p'_{j_l-1}} \frac{ p_{i_1} \dots p_{i_r} }{ p_1 \dots p_k p'_1 \dots p'_k }
\end{equation}
to \eqref{slo}.  Here we will keep only one of the conditions $1_{p_{i_l}|p_1+\dots+p_{i_l-1} - p'_1 - \dots - p'_{j_l-1}}$, and also estimate the $w_\gamma$ by \eqref{wg}, arriving at an upper bound of
\begin{equation}\label{gm}
 \ll \frac{\exp(O(k))}{(\log\log X)^{2k}} 
\sum_{p_1,\dots,p_k,p'_1,\dots,p'_k \in I} \frac{\prod_{l=1}^r (p_{i_l} 1_{p_{i_l}=p'_{j_l}})}{ p_1 \dots p_k p'_1 \dots p'_k }
1_{p_{i_{l_0}}| p_1 + \dots + p_{i_{l_0}-1} - p'_1 - \dots - p'_{j_{l_0}-1}}.
\end{equation}
We sum first over $p_{i_{l_0}-1}$, keeping all the other variables in $p_1,\dots,p_k,p'_1,\dots,p'_k$ fixed.  Then $p_{i_{l_0}-1}$ is constrained to a single residue class $a \hbox{ mod } p_{i_{l_0}}$. We can crudely bound
\[
\sum_{\substack{p_{i_{l_0}-1} \in I \\ p_{i_{l_0}-1} = a \hbox{ mod } p_{i_{l_0}} }} \frac{1}{p} \leq \sum_{\substack{n \in I \\ n = a \hbox{ mod } p_{i_{l_0}} }} \frac{1}{n} \ll \frac{\log X}{\exp(\sqrt{\log X})} \ll \exp( - (1+o(1)) \sqrt{\log X} )
\]
since $n, p_{i_{l_0}} \geq \exp(\sqrt{\log X})$ (we could have done slightly better using the Brun-Titchmarsh inequality, but this is not necessary here). 

The factor $\frac{\exp(O(k))}{(\log\log X)^{2k}}$ in \eqref{gm} is $\exp( o( \sqrt{\log X} ) )$ and so can be absorbed into the $o(1)$ error in the preceding estimate, arriving at an upper bound of
$$
\ll
\exp( - (1+o(1)) \sqrt{\log X} ) \sum_{p_1,\dots,p_{i_{l_0}-2}, p_{i_{l_0}},\dots,p_k,p'_1,\dots,p'_k \in I}
\frac{\prod_{l=1}^r (p_{i_l} 1_{p_{i_l}=p'_{j_l}})}{ p_1 \dots p_{i_{l_0}-2} p_{i_{l_0}} \dots p_k p'_1 \dots p'_k }.$$
Each of the variables $p_j$ or $p'_j$ that is not of the form $p_{i_l}$ or $p_{j_l}$ for some $l$ can be summed using 
\begin{equation}\label{as}
\sum_{p \in I} \frac{1}{p} \ll \log\log X,
\end{equation} 
and then noting that $O(\log\log X)^{2k} \ll \exp( o( \sqrt{\log X} ) )$, we arrive (after using the constraints $p_{i_l} = p'_{j_l}$ to collapse the sum) at
$$
\ll
\exp( - (1+o(1)) \sqrt{\log X} ) \sum_{p_{i_1},\dots,p_{i_r} \in I} \frac{1}{p_{i_1} \dots p_{i_r}},$$
and a further application of \eqref{as} then gives a total contribution of $\exp( - (1+o(1)) \sqrt{\log X} )$ for \eqref{gm}.  Finally, the total number of choices for $r$ and $(i_1,j_1), \dots, (i_r,j_r)$ may be crudely bounded by $k \times k^{2k}$, so we have a net bound of
$$ k \times k^{2k} \times \exp( - (1+o(1)) \sqrt{\log X} ) = o(1).$$
This concludes the proof of \eqref{meander}.

Finally, we prove \eqref{meander-2}.  We consider first the contribution of the case where there are no collisions, so that $\{p_1,\dots,p_k\}$ is disjoint from $\{p'_1,\dots,p'_k\}$.  In this case, we have
$$
\P( \gamma, \gamma' \subset {\mathbf G}_X )  = \frac{1}{p_1 \dots p_k p'_1 \dots p'_k} $$
and so this contribution to \eqref{meander-2} is bounded by
$$
 \sum_{\gamma,\gamma' \in \Gamma: \gamma(k)=\gamma'(k)} \frac{w_\gamma w_{\gamma'}}{p_1 \dots p_k p'_1 \dots p'_k}.$$
One can bound this by the probability that $\gamma(k) = \gamma'(k)$, where $\gamma, \gamma'$ are selected as in the proof of \eqref{wsum}.  But if we fix the variables $p'_1,\dots,p'_k,p_1,\dots,p_{k-1}$, then the constraint $\gamma(k) = \gamma'(k)$ is only satisfiable for a single value $p_k^0$ of $p_k$ in $I$, and so the probability here can be bounded using Lemma \ref{chic} by
$$ \ll \frac{1}{p_k^0 \log\log X} \ll \frac{1}{\exp( \sqrt{\log X} )}$$
which is acceptable.

Now we consider the contribution of the case where there is at least one collision.  By the computations used to prove \eqref{meander}, this contribution is bounded by
\begin{align*}
& \sum_{r=1}^k \sum_{(i_1,j_1),\dots,(i_r,j_r)} \sum_{\gamma,\gamma' \in \Gamma: \gamma(k) = \gamma'(k)} w_\gamma w_{\gamma'} \prod_{l=1}^r 1_{p_{i_l}=p'_{j_l}} 1_{p_{i_l}|p_1+\dots+p_{i_l-1} - p'_1 - \dots - p'_{j_l-1}} \\
&\quad \frac{ p_{i_1} \dots p_{i_r} }{ p_1 \dots p_k p'_1 \dots p'_k }.
\end{align*}
If one does not have $i_l=l$ for all $l=1,\dots,r$, then we can discard the $\gamma(k)=\gamma'(k)$ constraint and use the computations used to prove \eqref{meander} to obtain a bound of $\exp( - (1+o(1)) \sqrt{\log X} )$, which is acceptable.  Thus we may assume that $i_l=l$ for all $l=1,\dots,r$.  

Now suppose that $r<k$, so that $p_k$ is not one of the $p_{i_l}'$.  If we fix $p'_1,\dots,p'_k,p_1,\dots,p_{k-1}$, then as before the constraint $\gamma(k)=\gamma'(k)$ is satisfiable for only a single value of $p_k$.  Repeating the argument used to control \eqref{gm}, we see that the contribution of this case is also $\exp( - (1+o(1)) \sqrt{\log X} )$, which is acceptable. Thus we may assume that $r=k$.   The constraint $\gamma(k)=\gamma'(k)$ is now automatic, and we simplify this contribution to
\begin{equation}\label{sas}
\sum_{(1,j_1),\dots,(k,j_k)} \sum_{\gamma,\gamma' \in \Gamma} w_\gamma w_{\gamma'} \prod_{l=1}^k \frac{1_{p_l = p'_{j_l}} 1_{p_l | p_1 + \dots + p_{l-1} - p'_1 - \dots - p'_{j_l-1}}}{p_l}
\end{equation}
where the outer sum is over permutations $(j_1,\dots,j_k)$ of $(1,\dots,k)$.

Suppose that $j_l \neq l$ for some $1 \leq l \leq k$.  Let $l_0$ be the least such $l$, so that $j_l=l$ for $l < l_0$ and $j_{l_0} > l_0$.  Then the constraint
$$ p_{l_0} | p_1 + \dots + p_{l_0-1} - p'_1 - \dots - p'_{j_{l_0}-1}$$
simplifies (using $p_{l} = p'_{j_l}$) to
$$ p_{l_0} | p'_{l_0} + \dots + p'_{j_{l_0}-1} $$
which is a non-trivial constraint since $j_{l_0} > l_0$.  In particular, if we fix $p_1,\dots,p_{l_0-1},p_{l_0+1},\dots,p_k$, and hence all of the $p'_j$ except $p'_{j_{l_0}}$, we see that $p_{l_0}$ is constrained to be a prime factor of a number of size at most $k X^{1/200}$; since $p_{l_0}$ lies in $I$, we see that there are at most $O( \sqrt{\log X} )$ choices for $p_{l_0}$, and the total sum of $\frac{1}{p_{l_0}}$ across these choices is thus $O( \sqrt{\log X} \exp( - \sqrt{\log X} ) )$.  Using this and Lemma \ref{chic}, and discarding all the other constraints $1_{p_l | p_1 + \dots + p_{l-1} - p'_1 - \dots - p'_{j_l-1}}$, we bound the contribution of \eqref{sas} of a single such $(1,j_1),\dots,(k,j_k)$ (which determines $l_0$) as
$$
\exp( -(1+o(1)) \sqrt{\log X} ) \sum_{p_1,\dots,p_{l_0-1},p_{l_0+1},\dots,p_k \in I} \frac{1}{p_1 \dots p_{l_0-1} p_{l_0+1} \dots p_k} $$
which by \eqref{as} is bounded by $O(\log\log X)^k \exp( -(1+o(1)) \sqrt{\log X} ) \leq \exp( -(1+o(1)) \sqrt{\log X} ) $ which is acceptable.
 
The only remaining case occurs when $j_l=l$ for all $1 \leq l \leq k$, at which point $\gamma'$ is equal to $\gamma$, the constraints $p_l | p_1 + \dots + p_{l-1} - p'_1 - \dots - p'_{j_l-1}$ can be discarded, and the contribution to \eqref{sas} collapses to
$$ \sum_{\gamma \in \Gamma} w_\gamma^2 \frac{1}{p_1 \dots p_k}.$$
This can be bounded by the expectation of $w_\gamma$, where $\gamma$ is chosen as in the proof of \eqref{wsum}.  But from Lemma \ref{chic}, this expectation is at most $\exp( O(k) ) / (\log\log X)^k$, which from the choice \eqref{k-def} of $k$ is (barely) $O( \log^{-100} X )$, which is acceptable.  This proves \eqref{meander-2}, and tracing back all the preceding reductions we finally arrive at Theorem \ref{main-2}.

\section{A Vinogradov type result}

In this section we prove the following result of Vinogradov type which was needed in the proof of Lemma \ref{le:ABconnect}.

\begin{proposition} \label{prop:Vinogradov}  Let $X$ be large, let $I_1,I_2,I_3$ denote the intervals
$$ I_1 := (X,3X]; \quad I_2 := (5X, 7X]; \quad I_3 := (3X, 5X]$$
and let $m \in [-X,X]$ be odd.  Let $A$ be an integer.  Then there are $\asymp X^2 / (\log X)^3$ triples $(p_1,p_2,p_3)$ of primes with $p_1 \in I_1, p_2 \in I_2, p_3 \in I_3$ such that
$$ m = - p_1 + p_2 - p_3 $$
and such that $A - p_1, A - p_1 + p_2$ are not divisible by $p^2$ for any $p \leq w$.
\end{proposition}

The proof of Proposition \ref{prop:Vinogradov} relies crucially on the following 
slight modification of Vinogradov's result on representing odd integers
as sums of three primes.

\begin{lemma} \label{lem:Vinogradov}
We have,  for $m$ odd and fixed square-free $k$, and any $a_1,a_2$,
\begin{align*}
\sum_{\substack{m = -p_1 + p_2 - p_3 \\ p_j \in I_j \\ p_1 \equiv a_1 \pmod{k^2} \\ p_2 \equiv a_2 \pmod{k^2}}} 1 
= \frac{\mathcal{G}(m)\mathfrak{S}(m)}{(\log X)^3} \cdot \mathbf{1}_{(k, -a_1 + a_2 - m) = (k,a_1) = (k,a_2) = 1} 
\cdot \frac{f((k, m)) g(k)}{k \varphi(k)^3} + o \Big ( \frac{X^2}{(\log X)^3} \Big ),
\end{align*}
where $\mathcal{G}(m) = \# \{ (n_1, n_2, n_3) \in I_1 \times I_2 \times I_3 : m = -n_1 + n_2 - n_3 \}$ and $f, g$
are multiplicative functions such that \begin{align*}
f(p^{\alpha}) = f(p) & = \Big ( 1 + \frac{1}{(p-1)^3} \Big ) \cdot \Big ( 1 - \frac{1}{(p-1)^2} \Big )^{-1} \\
g(p^{\alpha}) = g(p) & = \Big ( 1 + \frac{1}{(p-1)^3} \Big )^{-1}.
\end{align*}
Finally 
$$\mathfrak{S}(m) := \prod_{p|M} (1 - \frac{1}{(p-1)^2}) \times \prod_{p \not | M} (1 + \frac{1}{(p-1)^3})$$ 
is the usual singular series appearing in Vinogradov's three-primes
theorem. 
\end{lemma}
\begin{proof}
This follows from a very minor modification of a generalization of Vinogradov's result due to Ayoub \cite{Ayoub}
(we only need to handle the additional condition $p_j \in I_j, \forall j \leq 3$). 
\end{proof}

One can easily compute that ${\mathcal G}(m) \asymp X^2$ and ${\mathfrak S}(m), f((k,m)), g(k) \asymp 1$.  Thus we have a cruder version
\begin{equation}\label{ayoub-weak}
\sum_{\substack{m = -p_1 + p_2 - p_3 \\ p_j \in I_j \\ p_1 \equiv a_1 \pmod{k^2} \\ p_2 \equiv a_2 \pmod{k^2}}} 1 
\asymp \frac{X^2}{\log^3 X} \Big ( \frac{1}{k \varphi(k)^3} + o(1) \Big ) 
\end{equation}
of the above lemma, when $a_1,a_2, -a_1+a_2-m$ are all coprime to $k$.  This is the only consequence of the above lemma that we will need.

We are now ready to prove Proposition \ref{prop:Vinogradov}.
\begin{proof}[Proof of Proposition \ref{prop:Vinogradov}]
Note that the only properties of the integer $A$ which are relevant are its reductions modulo $p^2$ for $p \leq w$.  Thus, by the Chinese remainder theorem, we may assume that $0 \leq A < \prod_{p \leq w} p^2$.  We may rewrite the desired claim as
\begin{equation} \label{equ:equmain}
\sum_{\substack{m = - p_1 + p_2 - p_3 \\ p_j \in I_j}} \mu^2_w(A - p_1) \mu^2_w(A - p_1 + p_2) \asymp \frac{X^2}{\log^3 X}
\end{equation}
where $\mu^2_w(n)$ is the indicator function of integers not divisible by a $p^2 \leq w$, and $p_1,p_2,p_3$ are understood to be prime.

Notice that for any fixed $C$, we have
\begin{equation}
\label{equ:muw2decomp}
\mu_w^2(A - p_1) \mu_w^2(A - p_1 + p_2) \geq \mathbf{1}_{\substack{p^2 \nmid A - p_1 \\ p^2 \nmid A - p_1 + p_2 \\ \forall p \leq C}}
- \sum_{\substack{C \leq p \leq w \\ p^2 | A - p_1}} 1 - \sum_{\substack{C \leq p \leq w \\ p^2 | A - p_1 + p_2}} 1.  
\end{equation}
We view the last two terms on the right-hand side as contributing error terms to be upper bounded using sieves. The first error term contributes to the left hand side of~\eqref{equ:equmain}
\begin{equation} \label{equ:firsterror}
\begin{split}
& \sum_{C \leq p \leq w} \sum_{\substack{m = -p_1 + p_2 - p_3 \\ p_j \in I_j, \forall j \leq 3 \\ p^2 | A - p_1}} 1 \\
& \leq \sum_{C \leq p \leq w} \sum_{\substack{p_1 \equiv A \pmod{p^2} \\ p_1 \in I_1}} | \{ n \in I_2 \cap (m+p_1+I_3): n, n-m-p_1 \hbox{ prime}\}|.
\end{split}
\end{equation}
The requirement that $n,n-m-p_1$ are prime removes two residue classes mod $p$ for $p \leq X$ not dividing $m+p_1$, and one residue class mod $p$ for $p \leq X$ dividing $m+p_1$. A standard upper bound sieve (see e.g. \cite[Theorem 3.12]{hr}) then gives
\begin{align*}
| \{ n \in I_2 \cap (m+p_1+I_3): n, n-m-p_1 \hbox{ prime}\}| &\ll \frac{X}{\log^2 X} \prod_{p \leq X: p | m+p_1} (1 + \frac{1}{p} ) \\
&\ll \frac{X}{\log^2 X} \sum_{d|m+p_1} \frac{1}{d} \\
&\ll \frac{X}{\log^2 X} \sum_{d \ll \sqrt{X}: d|m+p_1} \frac{1}{d} 
\end{align*}
since $m+p_1 \ll X$ and we may pair $d$ with $\frac{m+p_1}{d}$.
Thus we may upper bound \eqref{equ:firsterror} by
$$ \ll \frac{X}{\log^2 X}  \sum_{C \leq p \leq w} \sum_{d \ll \sqrt{X}} \frac{1}{d} \sum_{\substack{p_1 \equiv A \pmod{p^2} \\ p_1 \in I_1: d|m+p_1}} 1.$$
Applying the Brun-Titchmarsh inequality, we may bound this by
$$ \ll \frac{X^2}{\log^3 X} \sum_{C \leq p \leq w} \sum_d \frac{1}{d \phi( [p^2, d] )} $$
where $[p^2,d]$ denotes the least common multiple of $p^2$ and $d$.
By Euler products the innermost sum is $O(1/p^2)$, and so this expression is $O\left( \frac{1}{C} \frac{X^2}{\log^3 X} \right)$.

Similarly, the second error term in~\eqref{equ:muw2decomp} contributes to the left hand side of~\eqref{equ:equmain}
\begin{equation} \label{equ:seconderror}
\begin{split}
& \sum_{C \leq p \leq w} \sum_{\substack{m = -p_1 + p_2 - p_3 \\ p_j \in I_j, \forall j \leq 3 \\ p^2 | A - p_1 + p_2}} 1 \\
& = \sum_{C \leq p \leq w} \sum_{\substack{p_3 \in I_3 \\ p^2 | A + m + p_3}} |\{ n \in I_1 \cap (I_2 - m - p_3): n, m+n+p_3 \hbox{ prime}\}|.
\end{split}
\end{equation}
As before, standard upper bound sieves give 
$$ |\{ n \in I_1 \cap (I_2 - m - p_3): n, m+n+p_3 \hbox{ prime}\}| \ll \frac{X}{\log^2 X} \sum_{d \ll \sqrt{X}: d | m+p_3} \frac{1}{d} $$
and an application of Brun-Titchmarsh as before shows that the second error term is also $O\left( \frac{1}{C} \frac{X^2}{\log^3 X} \right)$.

It remains to understand the contribution of the main term.  Observe that for each prime $p$, there are at least $(p-1)(p-2) > p^2 - 3p$ pairs of residue classes $a_1,a_2 \pmod{p}$ such that $a_1, a_2, -a_1+a_2-m$ are all coprime to $p$.  Thus, there are at least $p^4 - 3p^3$ pairs of residue classes $a_1, a_2 \pmod{p^2}$ such that $a_1,a_2, -a_1+a_2-m$ are coprime to $p$.  Of these pairs, there are at most $2p^2$ pairs such that one of $A-a_1$ or $A-a_1+a_2$ is divisible by $p^2$.  By the Chinese remainder theorem, setting ${\mathcal P} := \prod_{p \leq C} p$, we conclude that there are at least $\prod_{p \leq C} (p^4 - 3p^3 - 2p^2)$ residue classes $a_1, a_2 \pmod{{\mathcal P}^2}$ such that $a_1,a_2,-a_1+a_2-m$ are coprime to all primes $p \leq C$, and $A-a_1, A-a_1+a_2$ are not divisible by $p^2$ for any $p \leq C$.  For any such fixed tuple $(a_1, a_2)$ we apply \eqref{ayoub-weak} to conclude (for $X$ sufficiently large depending on $C$)  that
$$
\sum_{\substack{m = -p_1 + p_2 - p_3 \\ p_j \in I_j, \forall j \leq 3 \\ p_1 \equiv a_1 \pmod{\mathcal{P}^2} \\ p_2 \equiv a_2 \pmod{\mathcal{P}^2}}} 1 \gg \frac{1}{\mathcal{P} \varphi(\mathcal{P})^3} \cdot \frac{X^2}{\log^3 X}.
$$
Summing, we may thus lower bound the contribution of the main term (for $C$ large) by
\[
\begin{split}
\gg \frac{\prod_{p \leq C} (p^4 - 3p^3 - 2p^2)}{{\mathcal P} \varphi({\mathcal P})^3} \frac{X^2}{\log^3 X} &= \prod_{p \leq C} \left(\frac{p^4 - 3p^3 - 2p^2}{p(p-1)^3}\right) \frac{X^2}{\log^3 X} \\
&= \prod_{p \leq C} \left(\frac{(p-1)^3-5p+1}{(p-1)^3}\right) \frac{X^2}{\log^3 X} \gg 
\frac{X^2}{\log^3 X}
\end{split}
\]
with the implied constant uniform in $C$.
For $C$ large enough (and $X$ sufficiently large depending on $C$), this lower bound dominates the two error terms, and we obtain the claim.
\end{proof}

\end{document}